\documentclass[
final
]{dmtcs-episciences}

\usepackage[utf8]{inputenc}
\usepackage{graphicx}
\usepackage{color}
\usepackage{subfig}
\usepackage{amsmath}
\usepackage{tikz}
\usepackage{enumerate}

\usepackage{amssymb}

\usepackage[round]{natbib}

\newcommand{\T}{\mathcal T}
\newcommand{\s}{\sigma}

\renewcommand{\b}[1]{{#1}}
\renewcommand{\r}[1]{{\dot{#1}}}
\newcommand{\A}[1]{A_{a,b}(#1)}

\let\leq=\leqslant
\let\geq=\geqslant

\newtheorem{thm}{Theorem}[section]
\newtheorem{lem}[thm]{Lemma}
\newtheorem{prop}[thm]{Proposition}
\newtheorem{cor}[thm]{Corollary}
\newtheorem{conj}[thm]{Conjecture}
\newtheorem{rem}[thm]{Remark}
\newtheorem{defi}[thm]{Definition}

\author{Alice L.L. Gao\affiliationmark{1}
  \and Emily X.L. Gao\affiliationmark{1}
  \and Patxi Laborde-Zubieta\affiliationmark{2}
  \and Brian Y. Sun\affiliationmark{3}\thanks{Corresponding author.~Email: brianys1984@126.com.}}
\title[Corners in Tree-like Tableaux]{Enumeration of Corners in Tree-like Tableaux\thanks{This work was partially supported by the 973 Project, the PCSIRT Project of the Ministry
of Education and the National Science Foundation of China and the Scientific Research Program of the Higher Education Institution of
Xinjiang Uygur Autonomous Region (No.~XJEDU2016S032)}}
\affiliation{
  Center for Combinatorics, LPMC-TJKLC Nankai University, Tianjin 300071, P. R. China\\
  LaBRI - University of Bordeaux,
351, cours de la Libration F-33405 Talence cedex, France\\
College of Mathematics and System Science,
Xinjiang University, Urumqi, Xinjiang 830046, P. R. China}
\received{2016-3-12}

\revised{2016-10-2}

\accepted{2016-11-9}
\begin{document}
\publicationdetails{18}{2016}{3}{17}{1408}
\maketitle
\begin{abstract}
  In this paper, we confirm conjectures of
Laborde-Zubieta on the enumeration of corners in tree-like tableaux and in
symmetric tree-like tableaux. In the process, we also enumerate corners in
(type $B$) permutation tableaux and (symmetric) alternative tableaux. The proof
is based on Corteel and Nadeau's bijection between permutation tableaux and
permutations. It allows us to interpret the number of corners as a statistic
over permutations  that is easier to count. The type $B$ case uses the
bijection of Corteel and Kim between type $B$ permutation tableaux and signed
permutations. Moreover, we give a bijection between corners and runs of size 1
in permutations, which gives an alternative proof of the enumeration of
corners. Finally, we introduce conjectural polynomial analogues of these
enumerations, and explain the implications on the PASEP.
\end{abstract}

\section{Introduction}
\label{sec:in}

The partially asymmetric exclusion process (PASEP) is a model from statistical
mechanics, in which particles jump stochastically to the left or to the right,
the probability of hopping left is $q$ times the probability of hopping right.
Moreover, particles can enter from the left with probability $\alpha$ and exit
at the right with probability $\beta$. We can describe (see~\cite{MR2352041}) the
equilibrium state of the PASEP using permutation tableaux in \cite{postnikov2006total},
alternative tableaux in~\cite{viennot:hal-00385528} or tree-like tableaux in~\cite{MR3158273}. These
combinatorially equivalent objects have been the focus of intense research in
the recent years. For example, the reader can be referred to~\cite{MR2376110},~\cite{MR2460235},~\cite{MR3459908},~\cite{MR2771605}
and references therein for more information.
 One of the main reasons being that
they are in bijection with permutations. For each of these three tableaux, a
type $B$ version was also defined, \textit{e.g.}, see~\cite{MR3158273},~\cite{MR2771605},~\cite{MR2293088}, and the previous bijections
with permutations were extended to bijections with signed permutations.

\cite{MR3459908} showed that the corners in tree-like tableaux are
interpreted in the PASEP as the locations where a jump of particle is possible.
He started with the enumeration of occupied corners and obtained the following results.
\begin{prop}{{\cite[Theorem~3.2]{MR3459908}}}\label{prop-Z}
    The number of occupied corners in the set of tree-like tableaux of size $n$ is $n!$.
\end{prop}
\begin{prop}{{\cite[Theorem~3.7]{MR3459908}}}\label{prop-Z1}
    The number of occupied corners in the set of symmetric tree-like tableaux of size $2n+1$ is $2^n\cdot n!$.
\end{prop}

Regarding the unrestricted corners, he gave the following two conjectures.
\begin{conj}{{\cite[Conjecture~4.1]{MR3459908}}}\label{conj1}
The number of corners in the set of tree-like tableaux of size $n$ is $n!\times\frac{n+4}{6}$.
\end{conj}
\begin{conj}{{\cite[Conjecture~4.2]{MR3459908}}}\label{conj2}
The number of corners in the set of symmetric tree-like tableaux of size $2n+1$ is $2^n\times n!\times \frac{4n+13}{12}.$
\end{conj}

In this work, we give a proof of these conjectures. For the first one, through
bijections, we give relations between the number of corners in permutation
tableaux, alternative tableaux and tree-like tableaux. Then, using a bijection
due to~\cite{MR2460235}, we interpret the number of corners in
permutation tableaux as a statistic in permutations. By computing this
statistic in permutations of fixed size, we deduce the enumeration of corners
in each of the three kind of tableaux (Theorem~\ref{corner-number-A}). The
second conjecture is proven in the same way, which also gives us the
enumeration of corners in type $B$ permutation tableaux, in symmetric
alternative tableaux (Theorem~\ref{corner-number-B}). It should be noted that
~\cite{hitczenko2015corners} proved both
conjectures in a different way, using a probabilistic approach. Additionally,
we present a bijection between corners in tree-like tableaux and runs of size 1
in permutations, which answers to a question raised in~\cite{gao2015zubieta2}. Counting
corners in tree-like tableaux gives an information about the average number of
locations where a jump of particle is possible in the PASEP, if we set
$q=\alpha=\beta=1$. We give a conjectural $(a,b)$-analogue of this enumeration
which would generalise the result to the case where only $q$ is equal to 1. We
also conjecture an $x$-analogue for the enumeration of corners in symmetric
tree-like tableaux.

The paper is organised in the following way. In Section~\ref{sec-2} we give several
definitions, in particular we recall the definitions of the tableaux we will be
considering. Section~\ref{sec-3} presents the different bijections we need to relate
corners in tree-like tableaux with 
permutations. In parallel, we also
deal with the type $B$ case. Then (Section~\ref{sec-4}), we prove the two conjectures
and enumerate corners in the other types of tableaux. Moreover we give a bijection
between corners in tree-like tableaux and runs of size 1 in permutations.
Finally we give polynomial analogues of Conjecture~\ref{conj1} and
Conjecture~\ref{conj2}, and partially prove them.

\section{Preliminaries}\label{sec-2}
 First of all, to be self-contained in this paper, let us recall some necessary basic notions and introduce some notations in this section.

 \subsection{\texorpdfstring{$(k,n)$}{Lg}-diagrams}
Here we mainly adopt Cho and Park's terminologies in \cite{MR3319075}. For two
nonnegative integers $n$ and $k$ with $n\geq k,$ a $(k,n)$-\emph{diagram}
$\mathcal{D}$ (left subfigure of Figure~\ref{4-8-diagram}) is a left-justified
diagram of boxes in a $k\times (n-k)$ rectangle with $\lambda_i$ boxes in the
$i$-th row, where $\lambda_1\geq \lambda_2\geq\cdots\geq\lambda_k\geq 0$. The
integer $n$ is called the \emph{length} of $\mathcal{D}$, it is equal to the
number of rows plus the number of columns. Note that a $(k,n)$-diagram may have
empty rows or columns.
A \emph{shifted $(k,n)$-diagram} is a $(k,n)$-diagram together with a
stair-shaped array of boxes added above, where the $j$-th column (from the
left) has $(n-k+1-j)$ additional boxes for $j\in[n-k].$ We denote
$\mathcal{D}^*$ the shifted $(k,n)$-diagram obtained from a $(k,n)$-diagram
$\mathcal{D}$. In terms of the definitions in \cite{MR3319075}, $\mathcal{D}$ is
called the $(k,n)$-subdiagram of $\mathcal{D}^*$. The \emph{length} of
$\mathcal{D}^*$ is defined to be the length of its $(k,n)$-subdiagram. Among
the cells we added, the ones at the top of a column are called
\emph{diagonals}. An example of a shifted $(4,8)$-diagram is shown in the
middle of Figure~\ref{4-8-diagram}, its diagonals are pointed out in the right
Figure~\ref{4-8-diagram}.
\setlength{\unitlength}{1mm}
\begin{center}
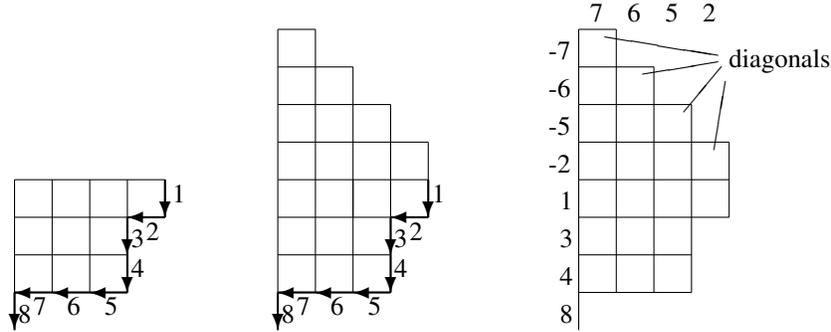
\begin{figure}[htbp]
\begin{picture}(60,40)
\multiput(25,0)(5,0){1}{\line(0,1){20}}

\multiput(30,5)(5,0){3}{\line(0,1){15}}
\multiput(45,15)(5,0){1}{\line(0,1){5}}
\multiput(25,5)(0,5){2}{\line(1,0){15}}
\multiput(25,15)(0,5){2}{\line(1,0){20}}
\put(25.5,1){8}\put(40.5,7){4}\put(40.5,11){3}
\put(46,17){1}\put(27.5,2){7}
{\thicklines
\put(45,20){\vector(0,-1){5}}
\put(45,15){\vector(-1,0){5}}
\multiput(40,10)(0,5){2}{\vector(0,-1){5}}
\multiput(40,5)(-5,0){3}{\vector(-1,0){5}}
\put(25,5){\vector(0,-1){5}}
}
\put(32,2){6}\put(37,2){5}\put(42.5,12){2}
\multiput(60,0)(5,0){1}{\line(0,1){20}}
\multiput(65,5)(5,0){3}{\line(0,1){15}}
\multiput(70,15)(5,0){1}{\line(0,1){5}}
\multiput(60,5)(0,5){2}{\line(1,0){15}}
\multiput(60,15)(0,5){2}{\line(1,0){20}}
\multiput(60,20)(5,0){2}{\line(0,1){20}}
\multiput(70,20)(5,0){1}{\line(0,1){15}}
\multiput(75,20)(5,0){1}{\line(0,1){10}}
\multiput(80,20)(5,0){1}{\line(0,1){5}}
\multiput(60,25)(0,5){1}{\line(1,0){20}}
\multiput(60,30)(0,5){1}{\line(1,0){15}}
\multiput(60,35)(0,5){1}{\line(1,0){10}}
\multiput(60,40)(0,5){1}{\line(1,0){5}}

\put(60.5,1){8}\put(75.5,7){4}\put(75.5,11){3}
\put(80.5,17){1}\put(62.5,2){7}
{\thicklines
\put(80,20){\vector(0,-1){5}}
\put(80,15){\vector(-1,0){5}}
\multiput(75,10)(0,5){2}{\vector(0,-1){5}}
\multiput(75,5)(-5,0){3}{\vector(-1,0){5}}
\put(60,5){\vector(0,-1){5}}
}
\put(67,2){6}\put(72,2){5}\put(77.5,12){2}

\multiput(100,0)(5,0){1}{\line(0,1){20}}
\multiput(105,5)(5,0){3}{\line(0,1){15}}
\multiput(120,15)(5,0){1}{\line(0,1){5}}
\multiput(100,5)(0,5){2}{\line(1,0){15}}
\multiput(100,15)(0,5){2}{\line(1,0){20}}
\multiput(100,20)(5,0){2}{\line(0,1){20}}
\multiput(110,20)(5,0){1}{\line(0,1){15}}
\multiput(115,20)(5,0){1}{\line(0,1){10}}
\multiput(120,20)(5,0){1}{\line(0,1){5}}
\multiput(100,25)(0,5){1}{\line(1,0){20}}
\multiput(100,30)(0,5){1}{\line(1,0){15}}
\multiput(100,35)(0,5){1}{\line(1,0){10}}
\multiput(100,40)(0,5){1}{\line(1,0){5}}
\put(97.5,1){8}\put(97.5,6){4}\put(97.5,11){3}
\put(97.5,16){1}\put(101.5,41){7}
\put(96,21){-2}\put(96,26){-5}\put(96,31){-6}
\put(96,36){-7}
\put(106.5,41){6}\put(111.5,41){5}\put(116.5,41){2}
\put(103.5,39){\line(6,-1){15}}
\put(108.5,34){\line(6,1){10}}
\put(114,29){\line(5,6){5}}
\put(118,24){\line(1,6){1.5}}
\put(120,35){diagonals}
\end{picture}
\caption{
   A $(4,8)$-diagram $\mathcal{D}$ (left),
    the shifted $(4,8)$-diagram $\mathcal{D}^*$ (middle)
    and the labeling of the rows and the columns of $\mathcal{D}^*$ (right).
}
\label{4-8-diagram}
\end{figure}
\end{center}

Let us now introduce some definitions and notations about those diagrams. A
$(k,n)$-diagram $\mathcal{D}$ is uniquely determined by its South-East border,
which is the lattice path starting at the North-East corner of the rectangle
$k\times(n-k)$, going along $\mathcal{D}$'s border and finishing at the
South-West corner. Following the same direction, we label the steps of the
South-East border with $[n]=\{1,\ldots,n\}$. We extend the labelling to shifted
$(k,n)$-diagrams. An example of both cases is given respectively in the left
and the middle subfigures of Figure~\ref{4-8-diagram}. The steps of the
South-East border are called, \emph{border edges}. We label the rows and the
columns of a $(k,n)$-diagram with the label of their corresponding vertical and
horizontal border edge respectively. Moreover, in the case of a shifted
$(k,n)$-diagram, we label the added rows as follows: if the diagonal cell of an
added row is in column $i$, then the row is labeled by $-i$. The right
subfigure of Figure~\ref{4-8-diagram} shows an example of the labelling of the
rows and columns of a shifted $(k,n)$-diagram. From now on, we will say
\emph{row} $i$ or \emph{column} $j$ when we actually refer to the row with the
label $i$ or to the column with the label $j$. The \emph{cell $(i,j)$} is the
cell at the intersection of the $i$th row starting from top and the $j$th
column starting from left, this notation is independent from the previous
labelling of rows and columns. The cells we will be looking at are the
followings.
\begin{defi}
    In a (shifted) $(k,n)$-diagram, a \emph{corner} is a cell such that its
    bottom and right edges are border edges.
\end{defi}
For example, the $(4,8)$-diagram in the left subfigure of
Figure~\ref{4-8-diagram} has two corners, $(1,4)$ and $(3,3)$. Its shifted
$(4,8)$-diagram has also two corners, $(5,4)$ and $(7,3)$, as we can see in the
middle subfigure of Figure \ref{4-8-diagram}. The last definition we need,
about $(k,n)$-diagrams, is the \emph{main diagonal}. As we can see in the
Figure~\ref{diagonal line}, it is the line going through the North-West and
the South-East corners of the top left cell.

\setlength{\unitlength}{1mm}
\begin{center}
\begin{figure}[htb]
\begin{picture}(40,16)
\multiput(70,-3)(5,0){1}{\line(0,1){20}}
\multiput(75,2)(5,0){2}{\line(0,1){15}}
\multiput(85,7)(5,0){1}{\line(0,1){10}}
\multiput(70,2)(0,5){1}{\line(1,0){10}}
\multiput(70,7)(0,5){2}{\line(1,0){15}}
\multiput(70,17)(0,5){1}{\line(1,0){20}}
\put(70,17){\line(1,-1){15}}
\put(86,2){\mbox{main diagonal line}}
\end{picture}
\caption{A diagram with its main diagonal line.}
\label{diagonal line}
\end{figure}
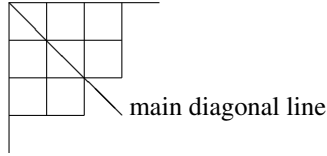
\end{center}

In this article we study alternative tableaux, permutation tableaux, tree-like
tableaux and their corresponding type $B$ versions. Tableaux should be understood
in the following way.
\begin{defi}\label{defi:tableau-shape-corner}
    Let $\mathcal{D}$ be a (shifted) $(k,n)$-diagram, a \emph{tableau} of underlying diagram $\mathcal{D}$, is a certain filling of the cells of $\mathcal{D}$ with some symbols. The \emph{underlying diagram} of $T$ is denoted by $\mathcal{D}(T)$. The previous definitions about (shifted) $(k,n)$-diagrams are extended to tableaux.
\end{defi}
A corner of tableaux $T$ is called an \emph{occupied corner} if it is filled with a
symbol, otherwise, a corner cell is called a \emph{non-occupied corner}. Let us
denote by $\mathcal{C}(T)$ the set of corners of a given tableau $T$ and
$\mathcal{C}(X)$ the set corners of a given set $X$ of tableaux, \textit{i.e.},
$$\mathcal{C}(X)=\bigcup_{T\in X}\mathcal{C}(T).$$ Similarly, denote by $c(T)$
the number of corners of $T$ and let $c(X)=|\mathcal{C}(X)|$.

\subsection{Tableaux}
In what follows, we shall introduce the main combinatorial objects of our work: permutation tableaux, alternative tableaux, tree-like tableaux and their type $B$ versions.

\subsubsection{Permutation tableaux}

Permutation tableaux arose in the study of totally nonnegative Grassmannian, see~\cite{postnikov2006total}. There have been a lot of work on the subject in many
different directions since they were formally introduced by~\cite{MR2293088}, see~\cite{MR2376110},~\cite{MR2460235},~\cite{MR2780856} and \cite{MR2352041} for details.

\begin{defi}
A \emph{permutation tableau}, is a $(k,n)$-diagram with no empty columns
together with a 0,1-filling of the cells such that
\begin{enumerate}[(1)]
    \item each column has at least one 1;
    \item there is no 0 which has a 1 above it in the same column and a 1
        to the left of it in the same row.
\end{enumerate}
\end{defi}

We denote by $\mathcal{PT}_n$ the set of permutation tableaux of length $n$. An
example of permutation tableau of length $8$ is given in the left subfigure of
Figure~\ref{permutation tableau}. We need to introduced some definitions about
permutation tableaux that will be needed in the description of the bijection
between alternative tableaux and permutation tableaux (Section~\ref{sec:ATtoPT}). In a
permutation tableau, a \emph{topmost 1} is a highest 1 of a column. A
\emph{restricted 0}, is a 0 with a 1 above it in the same column. Finally, a
\emph{rightmost restricted 0}, is a restricted 0 with no restricted 0 to its
right.

The type $B$ version of these tableaux, were introduced by~
\cite{MR2383586}.

\begin{defi}
A \emph{type $B$ permutation tableau} is a shifted $(k,n)$-diagram
$\mathcal{D}^*$ together with a $0,1$-filling of $\mathcal{D}^*$ satisfying the
following conditions:
\begin{enumerate}[(1)]
    \item each column has at least one 1;
    \item there is no 0 which has a 1 above it in the same column and a 1
        to the left of it in the same row;
    \item if a 0 is in a diagonal cell, then it does not have a 1 to the
        left of it in the same row.
\end{enumerate}
\end{defi}

We denote by $\mathcal{PT}_n^B$ the set of type $B$ permutation tableaux of
length $n$. An example of a permutation tableau of length $7$ is given in right
subfigure of Figure~\ref{permutation tableau}. We extend the definition of
topmost 1, restricted 0 and rightmost restricted 0, adding that a 0 in a
diagonal is restricted.

\setlength{\unitlength}{1mm}
\begin{center}
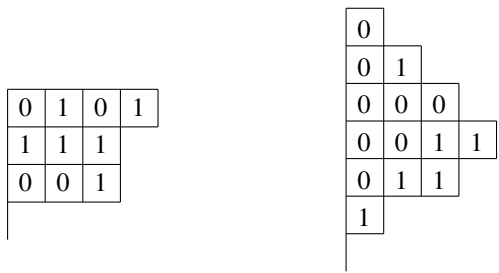
\begin{figure}[htbp]
\begin{picture}(0,35)(20,0)
\multiput(70,0)(5,0){1}{\line(0,1){20}}
\multiput(75,5)(5,0){3}{\line(0,1){15}}
\multiput(90,15)(5,0){1}{\line(0,1){5}}
\multiput(70,5)(0,5){2}{\line(1,0){15}}
\multiput(70,15)(0,5){2}{\line(1,0){20}}
\put(71.5,6.5){0}\put(76.5,6.5){0}\put(81.5,6.5){1}
\put(71.5,11.5){1}\put(76.5,11.5){1}\put(81.5,11.5){1}
\put(71.5,16.5){0}\put(76.5,16.5){1}\put(81.5,16.5){0}\put(86.5,16.5){1}
\end{picture}

\begin{picture}(0,0)(-60,0)

\multiput(35,0)(5,0){1}{\line(0,1){35}}
\multiput(40,5)(5,0){1}{\line(0,1){30}}
\multiput(45,10)(5,0){1}{\line(0,1){20}}
\multiput(50,10)(5,0){1}{\line(0,1){15}}
\multiput(55,15)(5,0){1}{\line(0,1){5}}

\multiput(35,5)(0,5){1}{\line(1,0){5}}
\multiput(35,10)(0,5){1}{\line(1,0){15}}
\multiput(35,15)(0,5){2}{\line(1,0){20}}
\multiput(35,25)(0,5){1}{\line(1,0){15}}
\multiput(35,30)(0,5){1}{\line(1,0){10}}
\multiput(35,35)(0,5){1}{\line(1,0){5}}

\put(36.5,6){1}\put(36.5,11){0}
\put(36.5,16){0}
\put(36.5,21){0}\put(36.5,26){0}
\put(36.5,31){0}

\put(41.5,11){1}\put(41.5,16){0}
\put(41.5,21){0}
\put(41.5,26){1}
\put(46.5,11){1}\put(46.5,16){1}
\put(46.5,21){0}
\put(51.5,16){1}
\end{picture}
\caption{A permutation tableau (left) and a type $B$ permutation tableau (right).}
\label{permutation tableau}
\end{figure}
\end{center}

\subsubsection{Alternative tableaux}
Alternative tableaux, were introduced by~\cite{viennot:hal-00385528} as follows.
\begin{defi}
An \emph{alternative tableau} is a $(k,n)$-diagram with a partial filling of
the cells with left arrows ``$\leftarrow$" and up arrows ``$\uparrow$", such
that all cells left of a left arrow ``$\leftarrow$", or above an up arrow
``$\uparrow$" are empty. In other words, all cells pointed by an arrow must be
empty.
\end{defi}

We denote by $\mathcal{AT}_n$ the set of alternative tableaux of length $n$.
An example of an alternative tableau of length $8$ is given in the left subfigure of
Figure~\ref{alternative tableau}.

\setlength{\unitlength}{1mm}
\begin{center}
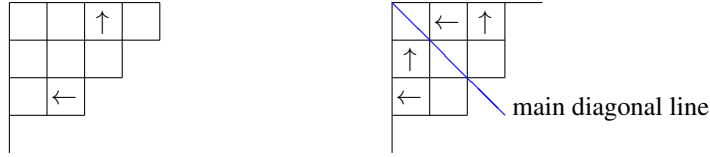
\begin{figure}[htbp]
\begin{picture}(40,20)(30,0)
\multiput(70,0)(5,0){1}{\line(0,1){20}}
\multiput(75,5)(5,0){2}{\line(0,1){15}}
\multiput(85,10)(5,0){1}{\line(0,1){10}}
\multiput(90,15)(5,0){1}{\line(0,1){5}}
\multiput(70,5)(0,5){1}{\line(1,0){10}}
\multiput(70,10)(0,5){1}{\line(1,0){15}}
\multiput(70,15)(0,5){2}{\line(1,0){20}}
\put(75.5,6.5){$\leftarrow$}
\put(81.5,16.5){$\uparrow$}
\end{picture}
\begin{picture}(40,15)(20,0)
\multiput(70,0)(5,0){1}{\line(0,1){20}}
\multiput(75,5)(5,0){2}{\line(0,1){15}}
\multiput(85,10)(5,0){1}{\line(0,1){10}}
\multiput(70,5)(0,5){1}{\line(1,0){10}}
\multiput(70,10)(0,5){2}{\line(1,0){15}}
\multiput(70,20)(0,5){1}{\line(1,0){20}}
\put(70,20){{\color{blue}\line(1,-1){15}}}
\put(86,5){\mbox{main diagonal line}}
\put(70.5,6.5){$\leftarrow$}\put(75.5,16.5){$\leftarrow$}
\put(71.5,11.5){$\uparrow$}\put(81.5,16.5){$\uparrow$}
\end{picture}
\caption{An alternative tableau (left) and a symmetric alternative tableau (right).}
\label{alternative tableau}
\end{figure}
\end{center}
\vskip -5mm

The type $B$ version of alternative tableau are called symmetric alternative tableaux, they were defined by \cite{MR2771605}.
\begin{defi}
A \emph{symmetric alternative tableau} is an alternative tableau unchanged by
the reflection with respect to its main diagonal.
\end{defi}
The set of symmetric alternative tableaux of length $2n$ will be denoted
$\mathcal{AT}_{2n}^{sym}$. A symmetric alternative tableau of size 8 is given in the
right subfigure of Figure~\ref{alternative tableau}.

\subsubsection{Tree-like tableaux}
The last kind of tableaux we will consider are the tree-like tableaux. They
were introduced in~\cite{MR3158273}. They have a nice
recursive structure, given by an insertion algorithm, which simplified some of
the previous main results.
\begin{defi}
A \emph{tree-like tableau} is a filling of $(k,n)$-diagram (without empty rows or empty columns) with points inside some cells, such that the resulting diagram satisfies the following three rules,
\begin{enumerate}[(1)]
    \item the top left cell of the diagram contains a point, called the
        \emph{root point};
    \item for every non-root pointed cell $c$, there exists either a
        pointed cell above $c$ in the same column, or a pointed cell to its
        left in the same row, but not both;
    \item every column and every row possess at least one pointed cell.
\end{enumerate}
\end{defi}
The \emph{size} of a tree-like tableau is defined to be its number of points.
It is not difficult to see that the length of a tree-like tableau is equal to its
size plus one. In the sequel, we denote by $\mathcal{T}_n$ the set
of the tree-like tableaux of size $n$. An example of a tree-like tableau of
size 8 is shown in the left subfigure of Figure~\ref{fig: tree-like example}.

\begin{defi}
A \emph{symmetric tree-like tableau} is a tree-like tableau unchanged by
the reflection with respect to its main diagonal.
\end{defi}

The size of a symmetric tree-like tableau is necessarily odd, we denote by
$\mathcal{T}^{sym}_{2n+1}$ the set of symmetric tree-like tableaux of size
$2n+1$. An example of a symmetric tree-like tableau of size $9$ is given in
the right subfigure of Figure~\ref{fig: tree-like example}.
\vspace{1cm}

\setlength{\unitlength}{1mm}
\begin{center}
\begin{figure}[htbp]
\begin{picture}(40,15)(30,0)
\multiput(70,0)(5,0){2}{\line(0,1){25}}
\multiput(80,5)(5,0){1}{\line(0,1){20}}
\multiput(85,10)(5,0){1}{\line(0,1){15}}
\multiput(90,15)(5,0){1}{\line(0,1){10}}
\multiput(70,0)(0,5){1}{\line(1,0){5}}
\multiput(70,5)(0,5){1}{\line(1,0){10}}
\multiput(70,10)(0,5){1}{\line(1,0){15}}
\multiput(70,15)(0,5){3}{\line(1,0){20}}
\put(71.5,21.5){$\bullet$}
\put(71.5,16.5){$\bullet$}
\put(76.5,16.5){$\bullet$}
\put(81.5,21.5){$\bullet$}
\put(86.5,16.5){$\bullet$}
\put(81.5,11.5){$\bullet$}
\put(76.5,6.5){$\bullet$}
\put(71.5,1.5){$\bullet$}
\end{picture}
\begin{picture}(40,15)(20,0)
\multiput(70,0)(5,0){2}{\line(0,1){25}}
\multiput(80,5)(5,0){2}{\line(0,1){20}}
\multiput(85,10)(5,0){2}{\line(0,1){15}}
\multiput(95,20)(5,0){1}{\line(0,1){5}}
\multiput(70,0)(0,5){1}{\line(1,0){5}}
\multiput(70,5)(0,5){1}{\line(1,0){15}}
\multiput(70,10)(0,5){2}{\line(1,0){20}}
\multiput(70,20)(0,5){2}{\line(1,0){25}}
\put(71.5,21.5){$\bullet$}
\put(71.5,11.5){$\bullet$}
\put(76.5,11.5){$\bullet$}
\put(81.5,16.5){$\bullet$}
\put(81.5,21.5){$\bullet$}
\put(91.5,21.5){$\bullet$}
\put(86.5,16.5){$\bullet$}
\put(76.5,6.5){$\bullet$}
\put(71.5,1.5){$\bullet$}
\put(70,25){{\color{blue}\line(1,-1){20}}}
\put(91,5){\mbox{main diagonal line}}
\end{picture}
\caption{A tree-like tableau (left) and a symmetric tree-like tableau (right).}
\label{fig: tree-like example}
\end{figure}
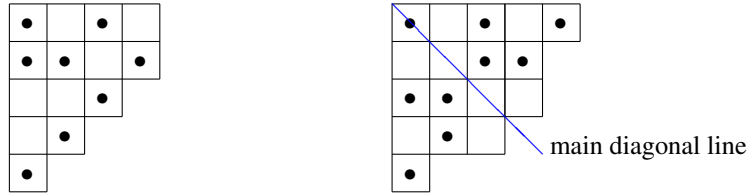
\end{center}

\subsection{Permutations}
A \emph{permutation} $\pi$ of length $n$ is a bijection from $[n]$ to $[n]$, we
use the notation $\pi_i:=\pi(i)$ for $1\leq i\leq n$. We can represent a
permutation with the word $\pi_1\ldots\pi_n$. The group of permutations of
length $n$ is denoted by $\mathfrak{S}_n$. We consider the following non usual
definitions for ascents and descents given in \cite{MR2460235}.

\begin{defi}
Given a permutation $\pi=\pi_1\cdots\pi_n\in \mathfrak{S}_n$ with the
convention that $\pi_{n+1}=n+1$, we say that $\pi_i$ is a \emph{descent} if
$\pi_i>\pi_{i+1}$ and call $\pi_i$ an \emph{ascent} if
$\pi_i<\pi_{i+1}$ for $1\leq i\leq n$.
\end{defi}

For example, the descents of the permutation $5,7,6,3,1,2,8,4$ are $\{7, 6, 3, 8\}$.

A \emph{signed permutation} (also called permutation of type $B$) $\sigma$ of
length $n$ is a bijection on $\{-n,-(n-1),\ldots,-1,1\ldots,n\}$ satisfying
$\sigma(-i)=-\sigma(i)$ for $i\in [n]$. We also use the notation $\sigma_i:=\sigma(i)$ and
represent a signed permutation with the word $\sigma_1\ldots\sigma_n$. The group of
type $B$ permutations is denoted by $\mathfrak{S}_n^B$. ~\cite{MR2780856} gave the following definitions of an ascent and a descent of a signed
permutation.
\begin{defi}
Let $\sigma=\sigma_1\sigma_2\cdots\sigma_n\in\mathfrak{S}_n^B$ with convention that
$\sigma_{n+1}=n+1$. For $i\in[n],$ $\sigma_i$ is a signed descent if $\sigma_i<0$ or
$\sigma_i>|\sigma_{i+1}|$, otherwise $\sigma_i$ is a signed ascent and satisfies
$0<\sigma_i<|\sigma_{i+1}|$.
\end{defi}

For $\sigma=3,-1,-4,2,6,5,7\in \mathfrak{S}_7^B,$ the signed descents of
$\sigma$ are $\{-1,-4,3,6\}.$

\section{Bijections}\label{sec-3}
In this section we give bijections between the different kinds of tableaux and we deduce equalities between the number of corners in each type of tableaux.

\subsection{A bijection \texorpdfstring{$\alpha$}{Lg} between tree-like tableaux and alternative tableaux compatible with the type \texorpdfstring{$B$}{Lg} case}
 Recall that $\mathcal{T}_n$ denote the set of tree-like tableaux of size $n$ and $\mathcal{AT}_n$ denote the set of alternative tableaux of length $n$.

\begin{thm}[{\cite{MR3158273}}]\label{thm:TT2AT}
There is a bijection $\alpha:\mathcal{T}_n\rightarrow \mathcal{AT}_{n-1}$ such
that for any $T\in \mathcal{T}_n$, the underlying diagram of $\alpha(T)$ is
obtained from $\mathcal{D(T)}$ by deleting its topmost row and its
leftmost column.
\end{thm}

We give a description of the bijection $\alpha$ and its inverse $\alpha^{-1}$
in detail without proof.
Given a tree-like tableau $T\in \mathcal{T}_n$ of size
$n$, we construct $\alpha(T)$ in two steps. First replace every non-root point
$p$ with a left arrow ``$\leftarrow$" if there is no point to its left in the
same row and an up arrow ``$\uparrow$" if there is no point above it in the
same column. Then, simply delete the topmost row and the leftmost column. One
can verify that the tableau we obtain is an alternative tableau of length
$n-1$. Figure~\ref{bijection-2} gives an example of the bijection.

Let $T'$ be an alternative tableau of length $n-1$ with underlying diagram
$\mathcal{D}'$. Suppose $\mathcal{D}'$ has $k$ rows and $n-1-k$ columns. We
construct the underlying diagram $D$ of $\alpha^{-1}(T')$ from $D'$ by adding a
column of $k$ cells at the left of its leftmost column, a row of $n-1-k$ cells
above its topmost row and a cell at its top left corner.
Note that $\mathcal{D}'$ is a
$(k,n-1)$-diagram and $\mathcal{D}$ is a $(k+1,n+1)$-diagram. Next, for any
cell $c=(i,j)$ in $\mathcal{D}'$, if there is an up arrow $``\uparrow"$ in $c$
and there is no arrows (both $``\leftarrow"$ and $``\uparrow"$ ) to its left
in the same row, add a point in the cell $(i+1,1)$ in $\mathcal{D}$. If
there is a left arrow $``\leftarrow"$ in $c$ and there is no arrows above $c$
in the same column, add a point in the cell $(1,j+1)$ in $\mathcal{D}$.
Then add a point in $(1,1)$ and $(1,j+1)$ or $(i+1,1)$ in $\mathcal{D}$ if
column $j$ or row $i$ has no arrows in $\mathcal{D}'$. Lastly, add a point in
$(i+1,j+1)$ in $\mathcal{D}$ if there is an arrow in $(i,j)$ in $\mathcal{D}'$,
the resulting tableau is a tree-like tableau of size $n$, denoted by
$T=\alpha^{-1}(T')$.

On the basis of the bijection, we can conclude the following result.

\begin{cor}\label{num-cor-TLT-AT}
    The number of corners in $\mathcal{T}_{n}$ and in $\mathcal{AT}_{n-1}$ satisfy the relation:
    $$
    c(\mathcal{T}_{n})= c(\mathcal{AT}_{n-1})+2(n-1)!.
    $$
\end{cor}

\begin{proof}
The underlying diagram of $\alpha(T)$ is obtained from $\mathcal D(T)$ by
removing the topmost row and the leftmost column. Hence, the bijection $\alpha$
doesn't create any new corner, but it removes the corners at the right of the
topmost row (corners of type 1) or the corners at the bottom of the leftmost
column (corners of type 2). Corners of type 1 in $\mathcal{T}_n$ are in easy
bijection with tree-like tableaux of size $n-1$, which are counted by $(n-1)!$.
To a corner of type 1 $c$ in a tree-like tableau $T$ of size $n$, we associate
the tree-like tableau obtained from $T$ by removing $c$. With a similar
argument, we can also prove that corners of type 2 in $\mathcal{T}_n$ are counted
by $(n-1)!$.
\end{proof}

Furthermore, the bijection $\alpha$ can be restricted to symmetric tree-like
tableaux and symmetric alternative tableaux. As an example, see
Figure~\ref{bijection-2}.

\begin{thm}\label{thm:sym TT2sym AT}
If $\alpha$ is restricted to the set of symmetric tree-like tableaux $\mathcal{T}_{2n+1}^{sym}$. Then it is also a bijection between the set of symmetric tree-like tableaux $\mathcal{T}^{sym}_{2n+1}$ and the set of symmetric alternative tableaux $\mathcal{AT}^{sym}_{2n}$.
\end{thm}

\setlength{\unitlength}{1mm}
\begin{center}
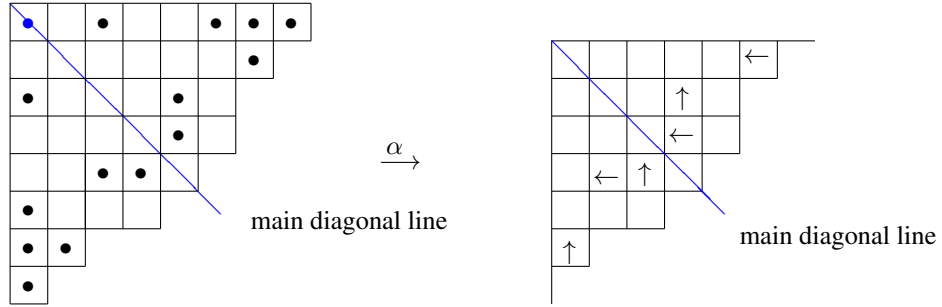
\begin{figure}[htbp]
\begin{picture}(40,40)
\multiput(100,0)(5,0){1}{\line(0,1){35}}
\multiput(105,5)(5,0){1}{\line(0,1){30}}
\multiput(110,10)(5,0){1}{\line(0,1){20}}
\multiput(115,10)(5,0){1}{\line(0,1){15}}
\multiput(120,15)(5,0){1}{\line(0,1){5}}
\put(105,35){\line(1,0){30}}
\put(110,30){\line(1,0){20}}
\multiput(110,30)(5,0){5}{\line(0,1){5}}
\put(115,25){\line(1,0){10}}
\put(120,20){\line(1,0){5}}
\multiput(115,25)(5,0){3}{\line(0,1){5}}
\multiput(120,20)(5,0){2}{\line(0,1){5}}

\multiput(100,5)(0,5){1}{\line(1,0){5}}
\multiput(100,10)(0,5){1}{\line(1,0){15}}
\multiput(100,15)(0,5){2}{\line(1,0){20}}
\multiput(100,25)(0,5){1}{\line(1,0){15}}
\multiput(100,30)(0,5){1}{\line(1,0){10}}
\multiput(100,35)(0,5){1}{\line(1,0){5}}
\put(125.5,32){$\leftarrow$}
\put(116.5,26.5){$\uparrow$}
\put(101.5,6){$\uparrow$}
{\color{blue}\put(100,35){\line(1,-1){23}}}
\put(115.5,21.5){$\leftarrow$}
\put(105.5,16){$\leftarrow$}

\put(111.5,16.3){$\uparrow$}

\put(77,18){$\longrightarrow$}
\put(78,20){$\alpha$}

\multiput(33,0)(5,0){1}{\line(0,1){35}}\put(33,40){\line(1,0){35}}
\put(28,40){{\color{blue}\line(1,-1){28}}}
\multiput(38,5)(5,0){1}{\line(0,1){25}}\put(38,35){\line(1,0){30}}
\multiput(43,10)(5,0){1}{\line(0,1){15}}\put(43,30){\line(1,0){20}}
\multiput(38,10)(5,0){1}{\line(0,1){10}}\put(48,25){\line(1,0){10}}
\multiput(38,40)(5,0){7}{\line(0,-1){5}}\put(53,20){\line(1,0){5}}
\multiput(43,35)(5,0){5}{\line(0,-1){5}}\put(58,20){\line(0,1){5}}
\multiput(48,30)(5,0){3}{\line(0,-1){5}}
\multiput(53,25)(5,0){1}{\line(0,-1){5}}
\multiput(33,5)(0,5){1}{\line(1,0){5}}
\put(48,10){\line(0,1){10}}
\multiput(33,10)(0,5){1}{\line(1,0){15}}
\multiput(33,15)(0,5){1}{\line(1,0){20}}
\multiput(33,20)(0,5){1}{\line(1,0){15}}
\multiput(33,25)(0,5){1}{\line(1,0){10}}
\multiput(33,30)(0,5){1}{\line(1,0){5}}
\put(34.5,6.5){$\bullet$}
\put(39.5,16.5){$\bullet$}

{\color{blue}\put(29.5,36.5){$\bullet$}}
\put(39.5,36.5){$\bullet$}\multiput(54.5,36.5)(5,0){3}{$\bullet$}
\put(59.5,31.5){$\bullet$}\multiput(49.5,26.5)(0,-5){2}{$\bullet$}

\put(29.5,26.5){$\bullet$}
\multiput(29.5,11.5)(0,-5){3}{$\bullet$}
\put(44.5,16.5){$\bullet$}
}
{
\put(33,35){\line(1,0){5}}\put(38,35){\line(0,-1){5}}
\put(38,30){\line(1,0){5}}\put(43,30){\line(0,-1){5}}
\put(43,25){\line(1,0){5}}\put(48,25){\line(0,-1){5}}
\put(48,20){\line(1,0){5}}\put(53,20){\line(0,-1){5}}
\put(28,0){\line(0,1){40}}\multiput(28,40)(0,-5){9}{\line(1,0){5}}
\put(33,40){\line(0,-1){5}}

\put(60,10){\mbox{main diagonal line}}
\put(125,8){\mbox{main diagonal line}}
\end{picture}
\caption{The bijection $\alpha$ between symmetric tree-like tableaux and symmetric alternative tableaux.}\label{bijection-2}
\end{figure}
\end{center}

With  Theorem \ref{thm:sym TT2sym AT}, it is natural to deduce the following corollary.

\begin{cor}\label{cor:cornersSYMTT-SYMAT}
The number of corners in $\mathcal{T}^{sym}_{2n+1}$ is equal to the number of corners in $\mathcal{AT}^{sym}_{2n}$ plus $2^n(n-1)!$, i.e.,
$$
 c(\mathcal{T}^{sym}_{2n+1})= c(\mathcal{AT}^{sym}_{2n})+2^n(n-1)!.
$$
\end{cor}
\begin{proof}
As in Corollary~\ref{num-cor-TLT-AT}, the bijection $\alpha$ doesn't create any
new corner, but it removes the corners at the right of the topmost row (type 1)
and the corners at the bottom of the leftmost column (type 2). A symmetric
tree-like tableau has a corner of type 1 if and only if it has a corner of type
2. Hence we will enumerate pairs of corners of type 1 and 2 belonging to the
same symmetric tree-like tableaux of size $n$. Such pairs are in easy bijection
with symmetric tree-like tableaux of size $n-1$, which are counted by
$2^{n-1}(n-1)!$. The bijection simply consists in removing the two corners.
\end{proof}

\subsection{A bijection \texorpdfstring{$\gamma$}{Lg} between alternative tableaux and permutation tableaux}
\label{sec:ATtoPT}
In this subsection, we give a simple description of the bijection $\gamma$, the
reader can refer to \cite{viennot:hal-00385528} for more details about its proof.
\begin{thm}[\cite{viennot:hal-00385528}]\label{PT-AT}
There exists a bijection $\gamma: \mathcal{PT}_n\rightarrow \mathcal{AT}_{n-1}$
such that for any permutation tableau $PT\in \mathcal{PT}_n$, the underlying
 diagram of alternative tableau $\gamma(PT)$ is obtained from
$\mathcal{D}(PT)$ by removing its first row.
\end{thm}
Given a permutation tableau $PT$, $\gamma(PT)$ is computed in the following
way. First, change the topmost 1 in every column to $``\uparrow"$. Then,
transform every rightmost restricted 0 to $``\leftarrow"$.
Finally, delete all other 0s and 1s, and erase the first row. See Figure
\ref{permutationtableaux} as an example.

Conversely, let $AT$ be an alternative tableau of length $n-1$ with
underlying diagram $\mathcal{D}'$. Suppose $\mathcal{D}'$ has $k$ rows and
$n-1-k$ columns. The underlying diagram $\mathcal{D}$ of $\gamma^{-1}(AT)$ is
obtained by adding to $\mathcal{D}'$ a row of $n-1-k$ cells above its first row
. We change the filling, in the following way. For any cell $c=(i,j)$ in
$\mathcal{D}'$, if there is an up arrow $``\uparrow"$ (resp. $``\leftarrow"$)
in $c$, add a 1 (resp. 0) in the cell $(i+1,j)$ in $\mathcal{D}$. Then, if
there is no 1 in some column $j$ of $\mathcal{D}$, we add a 1 in
the cell $(1,j)$. Lastly, fill with 0s the empty cells
to the left of a 0 in the same row, or above a 1 in the same column, and then,
fill the rest of empty cells with 1s.

\setlength{\unitlength}{1mm}
\begin{center}
\begin{figure}[htbp]
\begin{picture}(40,18)
\multiput(45,0)(5,0){1}{\line(0,1){20}}
\multiput(50,5)(5,0){3}{\line(0,1){15}}
\multiput(65,15)(5,0){1}{\line(0,1){5}}
\multiput(45,5)(0,5){2}{\line(1,0){15}}
\multiput(45,15)(0,5){2}{\line(1,0){20}}
\put(46.5,5.5){0}\put(46.5,11.5){1}
\put(46.5,16.5){0}\put(51.5,16.5){1}
\put(51.5,11.5){1}\put(51.5,5.5){0}
\put(56.5,5.5){1}\put(56.5,11.5){1}\put(56.5,16.5){0}
\put(61.5,16.5){1}
\put(80,10){$\longrightarrow$}
\put(82,12){$\gamma$}
\multiput(95,0)(5,0){1}{\line(0,1){15}}
\multiput(100,5)(5,0){3}{\line(0,1){10}}
\multiput(95,5)(0,5){2}{\line(1,0){15}}
\multiput(95,15)(0,5){1}{\line(1,0){20}}
\put(96.5,11.5){$\uparrow$}\put(100.5,6.5){$\leftarrow$}
\put(106.5,11.5){$\uparrow$}
\end{picture}
\caption{The bijection $\gamma$ from a permutation tableau to an alternative tableau.}
\label{permutationtableaux}
\end{figure}
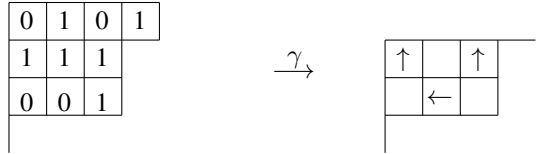
\end{center}

Based on Theorem \ref{PT-AT}, we can deduce the following corollary.
\begin{cor}\label{cor:num-corners-relation-AT-PT}
    The number of corners in $\mathcal{AT}_{n-1}$ and in $\mathcal{PT}_n$ satisfy the relation:
$$
c(\mathcal{AT}_{n-1})=c(\mathcal{PT}_{n})-(n-1)!.
$$
\end{cor}

\begin{proof}
The bijection $\gamma^{-1}$ doesn't form any corner, but if the step 1 is a west step, a new corner is constructed. The alternative tableau of length $n-1$ such that the step 1 is a west step, are in easy bijection with alternative tableaux of length $n-2$, which are counted by $(n-1)!$. The bijection simply consist in removing the west step.
\end{proof}
\subsection{A bijection \texorpdfstring{$\zeta$}{Lg} between  symmetric alternative tableaux  and type \texorpdfstring{$B$}{Lg} permutation tableaux }

The alternative representation of a permutation tableau of type $B$ was
introduced by~\cite{MR2780856}.

\begin{defi}
\label{def:alternative representation}
The \emph{alternative representation} of $PT^B\in \mathcal{PT}_n^B$ is a
tableau obtained from $PT^B$ according to the following operations, denoted by
$\mathcal{R}$. First, we replace the topmost 1s with $``\uparrow"$s and the
rightmost restricted 0s with $``\leftarrow"$s and remove the remaining 0s and
1s. Second, we remove the $``\uparrow"$s in the diagonal and cut off the
diagonal cells as shown in Figure~\ref{type B permutation tableau}. We call the
resulting tableau, denoted by $\mathcal{R}(PT^B)$, the alternative
representation of $PT^B$.
 \end{defi}

\setlength{\unitlength}{1mm}
\begin{center}
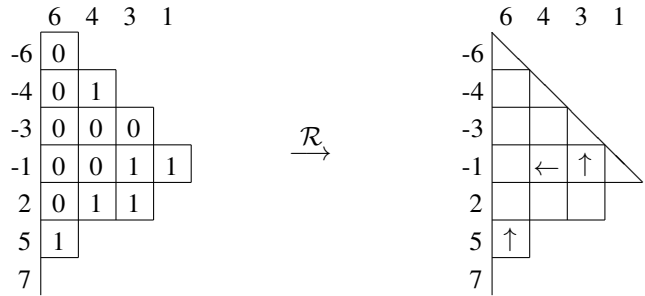
\begin{figure}[htbp]
\begin{picture}(40,38)
\multiput(35,0)(5,0){1}{\line(0,1){35}}
\multiput(40,5)(5,0){1}{\line(0,1){30}}
\multiput(45,10)(5,0){1}{\line(0,1){20}}
\multiput(50,10)(5,0){1}{\line(0,1){15}}
\multiput(55,15)(5,0){1}{\line(0,1){5}}

\multiput(35,5)(0,5){1}{\line(1,0){5}}
\multiput(35,10)(0,5){1}{\line(1,0){15}}
\multiput(35,15)(0,5){2}{\line(1,0){20}}
\multiput(35,25)(0,5){1}{\line(1,0){15}}
\multiput(35,30)(0,5){1}{\line(1,0){10}}
\multiput(35,35)(0,5){1}{\line(1,0){5}}
\put(32,1){7}\put(32,6){5}\put(32,11){2}
\put(31,16){-1}\put(31,21){-3}
\put(31,26){-4}\put(31,31){-6}
\put(36,36){6}\put(41,36){4}
\put(46,36){3}\put(51,36){1}

\put(36.5,6){1}\put(36.5,11){0}
\put(36.5,16){0}
\put(36.5,21){0}\put(36.5,26){0}
\put(36.5,31){0}

\put(41.5,11){1}\put(41.5,16){0}
\put(41.5,21){0}
\put(41.5,26){1}
\put(46.5,11){1}\put(46.5,16){1}
\put(46.5,21){0}
\put(51.5,16){1}
\put(68,18){$\longrightarrow$}
\put(69.5,20){$\mathcal{R}$}
\multiput(95,0)(5,0){1}{\line(0,1){35}}
\multiput(100,5)(5,0){1}{\line(0,1){25}}
\multiput(105,10)(5,0){1}{\line(0,1){15}}
\multiput(110,10)(5,0){1}{\line(0,1){10}}
\multiput(95,35)(5,0){1}{\line(1,-1){20}}

\multiput(95,5)(0,5){1}{\line(1,0){5}}
\multiput(95,10)(0,5){1}{\line(1,0){15}}
\multiput(95,15)(0,5){1}{\line(1,0){20}}
\multiput(95,20)(0,5){1}{\line(1,0){15}}
\multiput(95,25)(0,5){1}{\line(1,0){10}}
\multiput(95,30)(0,5){1}{\line(1,0){5}}

\put(92,1){7}\put(92,6){5}\put(92,11){2}
\put(91,16){-1}\put(91,21){-3}\put(91,26){-4}
\put(91,31){-6}
\put(96,36){6}\put(101,36){4}\put(106,36){3}
\put(111,36){1}

\put(96.5,6.5){$\uparrow$}
\put(100.5,16){$\leftarrow$}
\put(106.5,16.5){$\uparrow$}
\end{picture}
\caption{A type $B$ permutation tableau of length 7 (left) and its alternative representation (right). }\label{type B permutation tableau}
\end{figure}
\end{center}
\vskip -5mm

It is not difficult to see that $\mathcal{R}$ is a bijection, see \cite{MR2780856} for
details. We note that the alternative representation of a type $B$ permutation
tableau can be obtained by cutting a symmetric alternative tableau across its
main diagonal line. Hence we can construct a bijection $\zeta$ between
permutation tableaux of type $B$ and symmetric alternative tableaux via the
bijection $\mathcal{R}$.

\begin{thm}\label{thm:zeta-bijection-PTB-SYMAT}
There is a bijection $\zeta:\mathcal{PT}_n^B\rightarrow \mathcal{AT}^{sym}_{2n}$.
\end{thm}

\begin{proof}
We denote by $\mathcal{F}$ the reflection of an alternative representation
across all its diagonal cells (or main diagonal line). It is bijection between
alternative representations and symmetric alternative tableaux. Then
$\zeta:=\mathcal{F}\circ \mathcal{R}$, is a bijection from
$\mathcal{PT}_n^B$ to $\mathcal{AT}^{sym}_{2n}$.
\end{proof}

As a corollary of Theorem \ref{thm:zeta-bijection-PTB-SYMAT}, we have the following result.
\begin{cor}\label{cor:corners-symAT-PTB}
The number of corners in the set of symmetric alternative tableaux of length
$2n$ is equal to twice the number of corners in the set of type $B$ permutation
tableaux of length $n$ plus $2^{n-1}n!$, i.e.,
\begin{equation*}
    c(\mathcal{AT}^{sym}_{2n})= 2c(\mathcal{PT}_n^B)+2^{n-1}n!.
\end{equation*}
\end{cor}
\begin{proof}
The bijection $\zeta$ doesn't remove any corner, but it creates new ones. Let
$PT$ be a type $B$ permutation tableau. All the corners of $PT$ appears twice
in $\zeta(PT)$ and if the step 1 of $PT$ is a west step, an additional corner
appears in $\zeta(PT)$. To prove the Corollary, we just need to enumerate the
number of type $B$ permutation tableaux of size $n$ such that step 1 is a west
step, we denote $\mathcal{PT}_n^W$ the set of those tableaux. In fact, they are
in bijection with type $B$ permutation tableaux of size $n$ such that step 1 is
a south step ($\mathcal{PT}_n^S$). The bijection consist simply in removing the
rightmost diagonal, \textit{i.e.}, the cell above the step 1. Since the set of
permutation tableaux of size $n$ is the disjoint union of $\mathcal{PT}_n^W$
and $\mathcal{PT}_n^S$, hence,
$|\mathcal{PT}_n^W|=\frac{1}{2}|\mathcal{PT}_n|=2^{n-1}n!$.
\end{proof}

\subsection{A bijection \texorpdfstring{$\varphi$}{Lg} from permutation  tableaux to permutations}
To begin with, it is worthy to mention that there have been several bijections
from permutation tableaux to permutations up to now, see \cite{MR2376110, MR2460235,
MR2780856,MR2293088}. Here we only introduce the one due to Corteel and Nadeau. We give the
following theorem without describing the bijection, the reader can be referred to~\cite{MR2460235} for details.

\begin{thm}{{\cite[Theorem 1. (1)]{MR2460235}}}\label{lem:bij-PT-Per}
There exists a bijection $\varphi:\mathcal{PT}_n \rightarrow \mathfrak{S}_n $
such that for any permutation tableau $PT\in \mathcal{PT}_n$ and $1\leq i\leq
n$, $i$ is a label of a column in $PT$ if and only if $i$ is a
descent in $\pi$; and $i$ is a label of a row in $PT$ if and only if $i$ is an
ascent in $\pi$, where $\pi=\varphi(PT)$.
\end{thm}

By Theorem~\ref{lem:bij-PT-Per}, we can easily find the following result.

\begin{cor}\label{cor:bij-per-PT}
For a permutation tableau $PT\in \mathcal{PT}_n$ and $1\leq i < n$, the consecutive border edges labeled with $i$ and $i+1$ in the South-East border in its underlying diagram form a corner if and only if $i$ is an ascent and $i+1$ is a descent in $\pi=\varphi(PT)$.
\end{cor}

\begin{rem}
It should be noted that, since ascent and descent are not defined in the usual
way, the property such that $i$ is an ascent and $i+1$ is a descent, does not
correspond to a peak.
\end{rem}

\subsection{A bijection \texorpdfstring{$\xi$}{Lg} between type \texorpdfstring{$B$}{Lg} permutation tableaux and type \texorpdfstring{$B$}{Lg} permutations}
\cite{MR2780856} built a bijection between type $B$ permutation tableaux and signed permutations.
\begin{thm}{{\cite[Proposition 4.1]{MR2780856}}}\label{thm:bij-PTB-2-PerB}
There exists a bijection $\xi: \mathfrak{S}_n^B\rightarrow \mathcal{PT}_n^B$ such that for any $\sigma\in \mathfrak{S}_n^B$ and $1\leq i\leq n$, $i$ is a signed descent if any only if $i$ is a label of a column; and $i$ is a signed ascent if any only if $i$ is a row label.
\end{thm}

By Theorem \ref{thm:bij-PTB-2-PerB}, we deduce the following corollary easily.
\begin{cor}\label{cor:corners-typeBPT-STAT}
For a signed permutation $\sigma\in \mathfrak{S}_n^B$ and $PT^B=\xi(\sigma)$,
the consecutive border edges labeled with $i$ and $i+1$ in the underlying
shifted $(k,n)$-diagram of $PT^B$ form a corner if, and only if, $i$ is a signed
ascent and $i+1$ is a signed descent in $\sigma.$
\end{cor}

\section{Enumeration of Corners}\label{sec-4}
\subsection{Enumeration of Corners}

In this subsection, from the enumeration of the statistics on permutations and signed permutations that arose in Corollary~\ref{cor:bij-per-PT} and Corollary~\ref{cor:corners-typeBPT-STAT}, we deduce the enumeration of corners in each kind of tableau.

\begin{thm}\label{corner-number-A}
At fixed size, the number of corners in each of the three kinds of type $A$
tableaux, is counted by
\[c(\mathcal{PT}_n)=\left\{\begin{array}{cc}0,&\mathrm{~if~}n=1,\\
(n-1)!\times\frac{n^2+4n-6}{6},&\mathrm{~if~}n\geq2.\end{array}\right.\]
\vspace{2mm}
\[c(\mathcal{AT}_{n-1})=\left\{\begin{array}{cc}0,&\mathrm{~if~}n=1,\\
(n-1)!\times\frac{n^2+4n-12}{6},&\mathrm{~if~}n\geq2.\end{array}\right.\]
\vspace{2mm}
\[c(\mathcal{T}_{n})=\left\{\begin{array}{cc}1,&\mathrm{~if~}n=1,\\
n!\times\frac{n+4}{6},&\mathrm{~if~}n\geq2.\end{array}\right.\]
\end{thm}

\begin{proof}
For $n=1$, there is only one tableau for each kind, and only the
tree-like tableau of size 1 has a corner.

For general $n$, using Corollary~\ref{num-cor-TLT-AT},
Corollary~\ref{cor:num-corners-relation-AT-PT} and
Corollary~\ref{cor:bij-per-PT}, we just need to compute the number of $i'$s in
permutations of size $n$ such that $i$ is an ascent and $i+1$ is a descent.
Suppose $a_i(\pi)$ means that $i$ is an ascent and $i+1$ is a descent in a permutation $\pi$ and
\[\chi(a_i)=\left\{\begin{array}{cc} 1,& \mathrm{if~}a_i~\mathrm{is~true,}\\
0,& ~\mathrm{otherwise,}\end{array}\right.\]
then,
$$c(\mathcal{PT}_n)=\sum_{PT\in \mathcal{PT}_n}c(PT)
=\sum_{\pi\in \mathfrak{S}_n}\sum_{i=1}^{n-1}\chi(a_i(\pi))
=\sum_{i=1}^{n-1}\sum_{\pi\in \mathfrak{S}_n}\chi(a_i(\pi)).$$

For $1\leq i < n$, let $A_i$ denote the set of permutations in $\mathfrak{S}_n$ such that $i$ is an ascent and $i+1$ is a descent in $\pi\in\mathfrak{S}_n$, and $|A_i|$ the cardinality of $A_i$. It is clear that
$$|A_i|=\sum_{\pi\in \mathfrak{S}_n}\chi(a_i(\pi)).$$
So, it is sufficient to compute $|A_i|$ in order to count $c(\mathcal{PT}_n)$.

For any permutation $\pi=\pi_1\ldots\pi_n\in A_i$, suppose there exist $1\leq t_1,t_2 \leq n$ such that $\pi_{t_1}=i$ and $\pi_{t_2}=i+1$.
By the definition of ascents and descents, we know that $\pi_{t_1}<\pi_{t_1+1}$ and $\pi_{t_2}>\pi_{t_2+1}$. There are three cases to consider.
\begin{enumerate}[{Case  }1:]
\item If $t_2=t_1+1$, it means that there is a subsequence $\pi_{t_1},\pi_{t_2},\pi_{t_2+1}=i,i+1,\pi_{t_2+1}$ such that $i+1>\pi_{t_2+1}$ in $\pi$. It is easy to see that the number of such permutations is $(i-1)(n-2)!$.

\item If $t_1=t_2+1$, similarly there is a subsequence $\pi_{t_2},\pi_{t_1},\pi_{t_1+1}=i+1,i,\pi_{t_1+1}$ such that $i<\pi_{t_1+1}$ in $\pi$. It is clear that the number of such permutations is $(n-2)!+(n-i-1)(n-2)!=(n-i)(n-2)!$, where $(n-2)!$ counts the number of permutations such that $\pi_{t_1+1}=n+1$.
\item For $|t_1-t_2|>1$, there are two subcases to consider.
\begin{enumerate}
    \item if $\pi_{t_1+1}\leq n$, the number of such permutations is $(n-i-1)(i-1)(n-2)!$.
   \item if $\pi_{t_1+1}=n+1$, the number of such permutations is $(i-1)(n-2)!$.
\end{enumerate}

   Therefore, in total there are $(n-i)(i-1)(n-2)!$ such permutations in
   $\mathfrak{S}_n$.

\end{enumerate}
So, the number of corners in the set of permutation tableaux of length $n\geq 2$ is
\begin{align*}
c(\mathcal{PT}_n)&=\sum_{i=1}^{n-1}|A_i|\\
&=\sum_{i=1}^{n-1}[(i-1)(n-2)!+(n-i)(n-2)!+(n-i)(i-1)(n-2)!]\\
&=(n-2)!\sum_{i=1}^{n-1}[(i-1)+(n-i)+(n-i)(i-1)]\\
&=(n-2)!\sum_{i=1}^{n-1}[(n+1)i-i^2-1]\\
&=(n-1)!\times\frac{n^2+4n-6}{6}.\\
\end{align*}
This completes the proof.
\end{proof}

As the third author gave the number of occupied corners, see Proposition
\ref{prop-Z}, we can give an enumerative result for \emph{non-occupied corners}
in $\mathcal{T}_n$.
\begin{cor}\label{cor:noc}The number of non-occupied corners in $\mathcal{T}_n$ is $n!\times\frac{n-2}{6} $ for $n\geq 3$ and zero for $n=1,2$.
\end{cor}

We also obtain analogues results with type $B$ tableaux.

\begin{thm}\label{corner-number-B}
At fixed size, the number of corners in each of the three kind of type $B$
tableaux, is counted by
\[c(\mathcal{PT}_n^B)=\left\{\begin{array}{cc}0,&\mathrm{~if~}n=1,\\
 2^{n-1}(n-1)!\times\frac{4n^2+7n-12}{12},&\mathrm{~if~}n\geq2.\end{array}\right.\]
\vspace{2mm}
\[c(\mathcal{AT}_{2n}^{sym})=\left\{\begin{array}{cc}1,&\mathrm{~if~}n=1,\\
2^n(n-1)!\times\frac{4n^2+13n-12}{12},&\mathrm{~if~}n\geq2.\end{array}\right.\]
\vspace{2mm}
\[c(\mathcal{T}_{2n+1}^{sym})=\left\{\begin{array}{cc}3,&\mathrm{~if~}n=1,\\
2^nn!\times\frac{4n+13}{12},&\mathrm{~if~}n\geq2.\end{array}\right.\]
\end{thm}

\begin{proof}
For $n=1$, there are two tableaux for each kind. The two permutations tableaux
have no corners, only one of the two alternative tableaux has a corner and the
two tree-like tableaux have respectively one and two corners.

For general $n$, using Corollary~\ref{cor:cornersSYMTT-SYMAT},
Corollary~\ref{cor:corners-symAT-PTB} and Corollary
\ref{cor:corners-typeBPT-STAT}, we can compute the number of corners in each
kind of type $B$ tableaux, from the enumeration of $i'$s in the permutation
$\sigma$ such that $i(1\leq i < n)$ is a signed ascent and $i+1$ is a signed
descent in $\sigma$. Suppose, $b_i(\pi)$ means that $i$ is a signed ascent and
$i+1$ is a signed descent in  $\sigma$ and
\[\chi(b_i)=\left\{\begin{array}{cc} 1,& \mathrm{if~}b_i~\mathrm{is~true,}\\
0,& ~\mathrm{otherwise,}\end{array}\right.\]
then,
$$c(\mathcal{PT}_n^B)=\sum_{PT^B\in \mathcal{PT}_n^B}c(PT^B)
=\sum_{\sigma\in \mathfrak{S}_n^B}\sum_{i=1}^{n-1}\chi(b_i)
=\sum_{i=1}^{n-1}\sum_{\sigma\in \mathfrak{S}_n^B}\chi(b_i).$$

For $1\leq i < n$, let $B_i$ denote the set of type B permutations in $\mathfrak{S}_n^B$ such that $i$ is a signed ascent and $i+1$ is a signed descent in $\sigma\in\mathfrak{S}_n^B$, and $|B_i|$ the cardinality of $B_i$. It is clear that
$$|B_i|=\sum_{\sigma\in \mathfrak{S}_n^B}\chi(b_i).$$
So, it is sufficient to compute $|B_i|$ in order to count $c(\mathcal{PT}_n^B)$.

For $1\leq i < n$ and $\sigma=\sigma_1\sigma_2\cdots\sigma_n\in B_i$. Suppose there exist $1\leq t_1,t_2 \leq n$ such that $|\sigma_{t_1}|=i$ and $|\sigma_{t_2}|=i+1$.  It is worthy to mention that $1\leq t_1,t_2 \leq n$ and $t_1\neq t_2$.  By the definition of $B_i$ we know that $|\sigma_{t_1}|$ is a signed ascent and $|\sigma_{t_2}|$ is a signed descent in $\sigma$, which implies that
$0<\sigma_{t_1}<|\sigma_{t_1+1}|$ and $\sigma_{t_2}<0$ or $\sigma_{t_2}>|\sigma_{t_2+1}|$. So there are two cases to consider for a given integer $i$.
\begin{enumerate}[{Case  }1:]
\item If $\sigma_{t_1}=i<|\sigma_{t_1+1}|$ and $\sigma_{t_2}=-(i+1)$, we can divide it into two subcases:
    \begin{enumerate}[{Subase  }1:]
    \item $t_2=t_1+1$, which means that there is a subsequence $\sigma_{t_1},\sigma_{t_2}=i,-(i+1)$ in $\sigma$. By the definition of signed permutations, it is not difficult to compute the number of such signed permutations is $2^{n-2}(n-1)!$.
    \item $t_2\neq t_1+1$, then there are two subcases to consider.
          \begin{enumerate}[(i)]
          \item  if $|\sigma_{t_1+1}|\leq n$, the number of such type B permutations is $2^{n-2}(n-i-1)(n-1)!$;
          \item  if $\sigma_{t_1+1}=n+1$, the number of such type B permutations is $2^{n-2}(n-1)!$.
         \end{enumerate}
    Hence there are $2^{n-2}(n-i)(n-1)!$ such signed permutations in $\mathfrak{S}_n^B$.
    \end{enumerate}

\item If $\sigma_{t_1}=i<|\sigma_{t_1+1}|$ and $\sigma_{t_2}=i+1>|\sigma_{t_2+1}|$. Similarly, there are three subcases to consider:
\begin{enumerate}[{Subase  }1:]
\item if $t_2=t_1+1$, this implies that there is a subsequence
    $\sigma_{t_1},\sigma_{t_2},\sigma_{t_2+1}=i,i+1,\sigma_{t_2+1}$ such that
    $i+1>|\sigma_{t_2+1}|$ in $\sigma$. By the definition of permutations of
    type $B$, the number of such permutations is $2^{n-2}(i-1)(n-2)!$.
\item if $t_1=t_2+1$. That is to say, there is a subsequence
    $\sigma_{t_2},\sigma_{t_1},\sigma_{t_1+1}=i+1,i,\sigma_{t_1+1}$ such that
    $i<|\sigma_{t_1+1}|$ in $\sigma$. Analogously, the number of such
    permutations is $2^{n-2}(n-2)!+2^{n-2}(n-i-1)(n-2)!,$ where $2^{n-2}(n-2)!$
    counts the number of permutations such that $\sigma_{t_1+1}=n+1$. Thus
    such type $B$ permutations are counted by $2^{n-2}(n-i)(n-2)!$.
\item if $|t_1-t_2|>1$, there are two subcases to consider:
     \begin{enumerate}[(i)]
     \item  if $|\sigma_{t_1+1}|\leq n$, the number of such permutations is $2^{n-2}(n-i-1)(i-1)(n-2)!$.
     \item  if $\sigma_{t_1+1}=n+1$, the number of such permutations is $2^{n-2}(i-1)(n-2)!$.
     \end{enumerate}
     In total, there are $2^{n-2}(n-i)(i-1)(n-2)!$ such type $B$
     permutations.
\end{enumerate}
\end{enumerate}
All in all, the number of elements in $c(\mathcal{PT}_n^B)$ is given by
\begin{align*}
&~~\sum_{i=1}^{n-1}2^{n-2}\Big\{(n-1)!+(n-i)(n-1)!+(i-1)(n-2)!+(n-i)(n-2)!+(n-i)(i-1)(n-2)!\Big\}\\
&\qquad=2^{n-2}(n-1)!\sum_{i=1}^{n-1}\Big\{1+(n-i)\Big\}
+2^{n-2}(n-2)!\sum_{i=1}^{n-1}\Big\{(i-1)+(n-i)+(n-i)(i-1)\Big\}\\
&\qquad=2^{n-1}\times (n-1)!\times \frac{4n^2+7n-12}{12}.
\end{align*}
This completes the proof.
\end{proof}

By Proposition~\ref{prop-Z1} and Theorem~\ref{corner-number-B}, we can enumerate the non-occupied corners in symmetric tree-like tableaux of size $2n+1$.
\begin{cor}
 The number of non-occupied corners in symmetric tree-like tableaux $\mathcal{T}_{2n+1}^{sym}$ of size $2n+1$ is given by
 $$2^n\times n!\times\frac{4n+1}{12}.$$
 \end{cor}
\subsection{Bijection between corners and runs of size 1}\label{subsec-3.1}
In this subsection, we give an alternative proof of the enumeration of corners,
by constructing a bijection between corners in tree-like tableaux and ascending
runs of size 1 in permutations. This answers to a question raised in
\cite{gao2015zubieta2}. Ascending run is also called increasing run, which was first
studied deeply by~\cite{gessel1977generating}. Recently, Zhuang  studied further on
runs and generalized Gessel's results to allow for a much wider variety of
restrictions on increasing run lengths, for more details, see~\cite{MR3499494}.
There is a closed formula counting the number of ascending runs of size $r$ in
permutations of size $n$ \cite[A122843]{sloane2011line}, for $0<r<n$ we have
\begin{equation}\label{runs}
\frac{n!\cdot\left[(n(r(r+1)-1) - r(r-2)(r+2) + 1\right]}{(r+2)!}
\end{equation}
In particular, for $r=1$, we get the sequence enumerating corners in tree-like
tableaux.

An \emph{ascending run} of length $r$ of a permutation $\s=\s_1\cdots\s_n$,
is a sequence $(\s_m,\ldots,\s_{m+r-1})$ such that
$$\s_{m-1}>\s_m<\s_{m+1}<\cdots<\s_{m+r-2}<\s_{m+r-1}>\s_{m+r},$$
with the convention that $\s_{0}=n+1$ and $\s_{n+1}=0$.
In particular, $\s$ has a run of size 1 means that there exists $i\in[n]$
such that $\s_{i-1}>\s_{i}>\s_{i+1}$.

In order to build the bijection, we need a preliminary result about
\emph{non-ambiguous trees} due to~\cite{MR3194208}. They correspond to
rectangular shaped tree-like tableaux.
The \emph{height} (resp. \emph{width}) of a non-ambiguous tree
is its number of row (resp. column) minus 1.
We have the following result about these objects (it is a reformulation of
Proposition~1.16 in~\cite{aval2015non}).
\begin{prop}\label{prop:nat}
Non-ambiguous trees of height $h$ and width $w$ are in bijection
with permutations $\sigma$ of $\{\b{1},\b{2},\ldots,\b{w},
\r{0},\r{1},\ldots,\r{h}\}$,
finishing by a pointed element and such that
if two consecutive elements $\s_i$ and $\s_{i+1}$ are both pointed or not pointed,
then $\s_i<\s_{i+1}$.
\end{prop}
\begin{proof}
The initial result is that, non-ambiguous trees of height $h$ and width $w$
are in bijection with pairs $(u,v)$ of 2-colored words,
with blue letters on $[w]$ and red letters on $[h]$,
where each letter appear exactly once (in $u$ or in $v$),
letters in blocks of the same colors are decreasing,
$u$ (resp. $v$) ends by a red (resp. blue) letter.

In order to obtain the Proposition~\ref{prop:nat},
we turn pairs $(u,v)$ into the desired permutations $\sigma$.
Let us consider a non-ambiguous tree $nat$ and its corresponding pair
$(u,v)$. We start by constructing a pair $(u',v')$ by replacing
the blue (resp. red) letters $i$ of $u$ and $v$
by the uncolored (resp. uncolored pointed) letters $w-i+1$ (resp. $h-i+1$).
The permutation $\sigma$ corresponding to $nat$ is $v'\r{0}u'$.
\end{proof}

The bijection between corners and ascending runs of size 1
is decomposed into two steps: Lemma~\ref{lemme:coupe} and
Lemma~\ref{lemme:triplet_to_permutation}.

\begin{lem}\label{lemme:coupe}
For $n\geq1$, there is a bijection between corners in $\T_n$ and triplets $(T_l,T_r,nat)$ such that
\begin{itemize}
\item $T_l$ is a tree-like tableau of size $n_l$,
\item $T_r$ is a tree-like tableau of size $n_r$,
\item $n_l+n_r+1=n$,
\item $nat$ is a non-ambiguous tree of height $left(T_r)+1$ and width $top(T_l)+1$.
\end{itemize}
\end{lem}
\begin{proof}
The proof is based on the $l$-cut procedure defined in~\cite{MR3194208}.
Let $T$ be in $\T_n$ and $c$ one of its corners. We start by cutting $T$
along the lines corresponding to the bottom and the right edges of $c$,
as shown in Figure~\ref{fig:coupe_a}. We denote by $L$ the bottom part,
$M$ the middle part and $R$ the right part. In order to obtain
$T_l$ we add to $L$ a first row whose length is equal to the number of columns of $M$
minus 1. There is exactly one way to add dots in this first row
for $T_l$ to be a tree-like tableau: we put them inside the cells
corresponding to non empty columns in $M$. In a similar way, we obtain $T_r$ from $R$
by adding a dotted first column. $T_l$ and $T_r$ are two tree-like tableaux,
and the sum of their length is equal to the length of $T$, hence $n_l+n_r+1=n$.
Finally, removing the empty rows and columns of $M$ we obtain $nat$.
This procedure is illustrated in Figure~\ref{fig:coupe_b}.
It should be clear that the construction can be reversed and that
$T_l$, $T_r$ and $nat$ verifies the desired conditions.
\begin{figure}[htbp]
    \begin{center}
        \captionsetup[subfigure]{width=150pt}
        \subfloat[The cutting of $T$ defined by the corner $c$.]{\label{fig:coupe_a}\includegraphics[scale=.5]{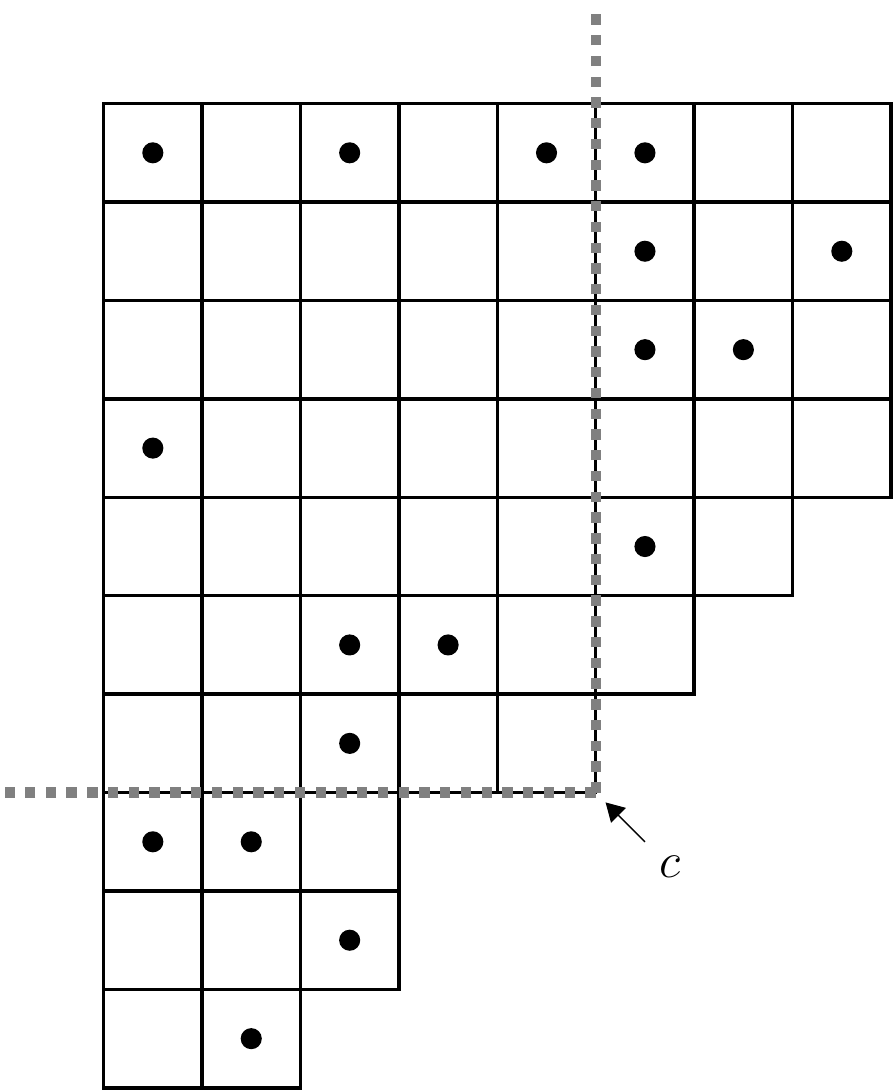}}
        \hfil\hfil\hfil\hfil
        \subfloat[The three parts obtained from $T$. (To obtain $nat$, the gray cells should be removed.)]{\label{fig:coupe_b}\includegraphics[scale=.5]{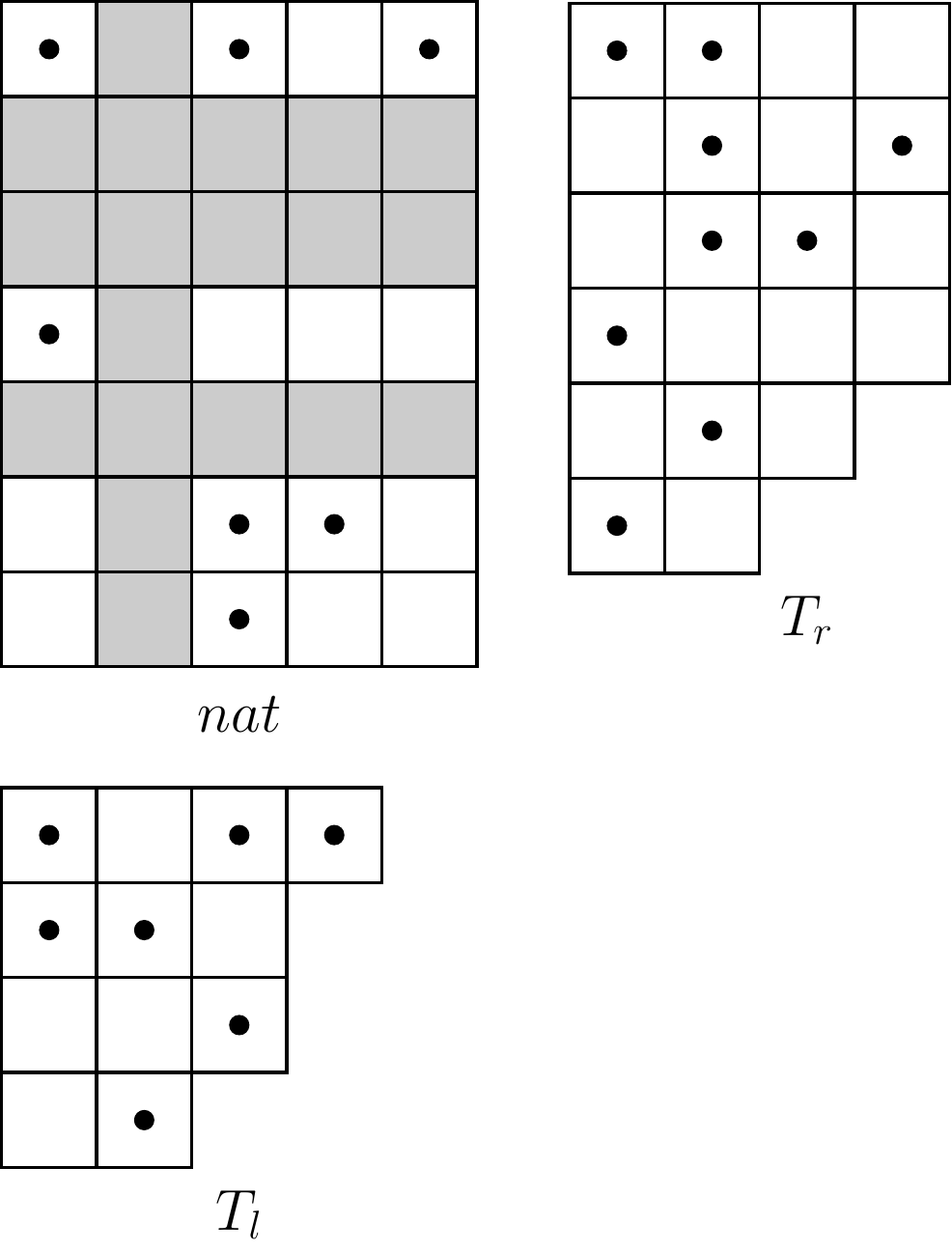}}

        \caption{An example of the bijection of Lemma~\ref{lemme:coupe}.}
        \label{fig:coupe}
    \end{center}
\end{figure}
\end{proof}

\begin{lem}\label{lemme:triplet_to_permutation}
The triplets $(T_l,T_r,nat)$ satisfying all the conditions in Lemma~\ref{lemme:coupe} are in bijection with runs of size 1 in permutations of size $n$.
\end{lem}
\begin{proof}
The idea of the proof is the following: from a triplet $(T_l,T_r,nat)$ we construct
a permutation $\s$ of size $n$ which have a run of size 1 $(\s_k)$ such that $\s_k=n_l+1$ for some $k$.
$T_l$ gives the ordering of the values smaller than $\s_k$,
$T_r$ the ordering of the values bigger than $\s_k$,
and $nat$ tells us how we mix them and where we put $\s_k$.

Tree-like tableaux of size $n$ with $k$ points in the first column are in bijection with
permutations of size $n$ with exactly $k$ cycles.
Indeed, by Proposition~1.3 in~\cite{MR3158273}, they
are in bijection with permutation tableaux of size $n$ with $k$
unrestricted rows. Moreover, by Theorem~1 in \cite{MR2460235}, they are in bijection with
permutations of size $n$ with $k$ right-to-left minimum.
Finally, the statistic of right-to-left minimum is equi-distributed
with the statistic left-to-right maximum, which is itself equi-distributed
with the statistic of the number of cycles as shown by the
``transformation fondamentale" of Foata-Sch\"utzenberger (Proposition~1.3.1 in~\cite{MR2868112}
or Section~1.3 in \cite{MR0272642}). Using an axial symmetry with respect to the main diagonal
of the underlying diagram of $T$, we deduce the same result for tree-like tableaux
of size $n$ with $k$ points in the first row.
We denote by $h$ and $w$ the height and the width of $nat$ respectively.
Let $l\s$ (resp. $r\s$) be the permutation associated to $T_l$ (resp. $T_r$)
by this bijection. We denote by $L_1,\cdots,L_w$ (resp. $R_1,\cdots,R_h$)
the disjoint cycles of $l\s$ (resp. $r\s$), such that if $i<j$, then
the maximum of $L_i$ (resp. $R_i$) is smaller than the maximum of $L_j$ (resp.$R_j$).
In addition, we shift the values of the $R_i$ by $n_l+1$.
We will write a cycle without parenthesis and with its biggest element
at the first position.
For example, suppose $l\s=(6)(7523)(9184)$ and $r\s=(423)(5)(716)(98)$, then
$L_1=6$, $L_2=7\text{ }5\text{ }2\text{ }3$, $L_3=9\text{ }1\text{ }8\text{ }4$,
$R_1=14\text{ }12\text{ }13$, $R_2=15$, $R_3=17\text{ }11\text{ }16$ and $R_4=19\text{ }18$.
Let $m$ be the word corresponding to $nat$.
If $\r{0}$ is not the last letter of $m$ and if the letter after $\r{0}$ is not pointed,
we can uniquely represent $m$ as $$m=u\r{0}\b{a_1}\r{b_1}\b{a_2}\r{b_2}\cdots\b{a_p}\r{b_p},$$
where $u$ can be empty,
and the words $\r{b_i}$ (resp. $\b{a_i}$) consist of a non empty increasing sequence
of pointed (resp. non pointed) letter.
In this case, we replace $m$ by swapping
subwords $\b{a_i}$ and $\r{b_i}$ for all $i$ ($1\leq i\leq p$), \textit{i.e.},
$$m^*=u\r{0}\r{b_1}\b{a_1}\r{b_2}\b{a_2}\cdots\r{b_p}\b{a_p}.$$
Obviously, this operation is bijective.

We finish the general construction by substituting
$n_l+1$ for $\r{0}$, $L_i$ for $\r{i}$ and $R_i$ for $\b{i}$.
For example, if $$m=\b{2}\b{3}\r{2}\r{3}\b{1}\b{4}\r{0}\r{1}$$ then we obtain the run
of size one $$15\text{ }17\text{ }11\text{ }16\text{ }7\text{ }5\text{ }2\text{ }3\text{ }9\text{ }1\text{ }8\text{ }4\text{ }14\text{ }12\text{ }13\text{ }19\text{ }18\text{ }\underline{10}\text{ }6.$$
Another example, if $$m=\r{1}\b{4}\r{0}\b{12}\r{2}\b{3}\r{3}$$ then
$$m^*=\r{1}\b{4}\r{0}\r{2}\b{12}\r{3}\b{3}$$ and thus we get
$$6\text{ }19\text{ }18\text{ }\underline{10}\text{ }7\text{ }5\text{ }2\text{ }3\text{ }14\text{ }12\text{ }13\text{ }15\text{ }9\text{ }1\text{ }8\text{ }4\text{ }17\text{ }11\text{ }16.$$
In order to reverse the construction from a run of size one $(\s_k)$,
we use the ``transformation fondamentale" in each maximal sequence of integers
smaller (resp. larger) than $\s_k$. This way, we are able to identify
the $L_i$ (resp. $R_i$), so that we can recover $l\s$, $r\s$ and $nat$.
For example, if we consider the run
$$\text{4 2 6 11 9 12 \underline{8} 3 7 1 5 10},$$
we obtain
$$
\underbrace{4\text{ }2}_{L_\r{2}}
\underbrace{6}_{L_\r{3}}
\underbrace{11\text{ }9}_{R_\b{2}}
\underbrace{12}_{R_\b{3}}
\underbrace{8}_{\ _\r{0}}
\underbrace{3}_{L_\r{1}}
\underbrace{7\text{ }1\text{ }5}_{L_\r{4}}
\underbrace{10}_{R_\b{1}},$$
hence
$$
l\s = (3)(4\text{ }2)(6)(7\text{ }1\text{ }5)
\text{, }
r\s = (2)(3\text{ }1)(4)
\text{, }
m^*=\r{2}\r{3}\b{2}\b{3}\r{0}\r{1}\r{4}\b{1},
m = \r{2}\r{3}\b{2}\b{3}\r{0}\b{1}\r{1}\r{4}
.$$
\end{proof}

As a consequence of Lemma~\ref{lemme:coupe} and Lemma~\ref{lemme:triplet_to_permutation},
we have the following theorem.
\begin{thm}\label{thm:corners_runs}
For $n\geq1$, corners in $\T_n$ are in bijection with runs of size 1 in $\mathfrak{S}_n$.
\end{thm}
Even if we send corners to runs of size 1, the two statistics does not have the same distribution.
For example, the permutation 321 have 3 runs of size 1 while a tree-like tableau of size 3 cannot have 3 corners.

From Theorem~\ref{thm:corners_runs} and Equation~\eqref{runs} we deduce the
enumeration of corners.
\begin{cor}
The number of corners in $\T_n$ is $n!\cdot\frac{n+4}{6}$ for $n\geqslant2$ and 1 for $n=1$.
\end{cor}

\subsection{Polynomial analogues}\label{subsec:pol_analogues}
In this subsection, we present two conjectures that generalise the
enumeration of corners in tree-like tableaux and in symmetric tree-like
tableaux, by giving polynomial analogues.

\subsubsection{(\texorpdfstring{$a$,$b$)}{Lg}-analogue of the average number of non-occupied corners in tree-like tableaux}
To refine the enumeration of corners, the two statistics over
tree-like tableaux we consider are: $top$ and $left$, that were defined in
\cite{MR3158273}. They count the number of non-root points in the first row and in
the first column, respectively. They are interesting statistics since they
correspond to the parameters $\alpha$ and $\beta$ respectively in the PASEP.

As explained in Section 4 of ~\cite{MR3459908}, computing the average number of corners
gives us the average number of locations where a particle may jump to the left
or to the right in the PASEP model, in the case $\alpha=\beta=q=1$ and $\delta=\gamma=0$.
Computing the ($a$,$b$)-analogue of average number of corners
$$c_n(a,b):=\sum_{T\in\T_n}c(T) \cdot w(T),$$
where $w(T) = a^{top(T)}b^{left(T)},$
would extend the result to the case $q=1$ and $\delta=\gamma=0$,
if we replace $a$ by $\alpha^{-1}$ and $b$ by $\beta^{-1}$.

The ($a$,$b$)-analogue of the average number of tree-like tableaux,
was computed in \cite{MR3158273}, it is equal to
$$T_n(a,b):=\sum_{T\in\T_n} w(T) = (a+b)(a+b+1)\cdots(a+b+n-2).$$
It turns out that the ($a$,$b$)-analogue of the average number of occupied corners
is also $T_n(a,b)$. In order to prove this, we just need to redo the short proof
of Section~3.2 in~\cite{MR3459908} via keeping track of left and top points.
As a consequence of this result, computing the ($a$,$b$)-analogue
for corners or for non-occupied corners, is equivalent.
In this section, we focus on non-occupied corners,
because their study seems easier.
We denote by $noc_n(a,b)$ the ($a$,$b$)-analogue of the average number
of non-occupied corners, \textit{i.e.},
$$noc_n(a,b):=\sum_{T\in\T_n}noc(T) w(T),$$
where $noc(T)$ is the number of non-occupied corners of $T$.
In particular, Corollary~\ref{cor:noc} implies that $noc_n(1,1)=n!\cdot\frac{n-2}{6}$.
Using an implementation of tree-like tableaux in~\cite{stein2015sagemath1},
the following conjecture has been experimentally confirmed until $n=10$.

\begin{conj}
For $n\geq 3$, the ($a$,$b$)-analogue of the enumeration of non-occupied corners is
$$
noc_n(a,b)
=
\left(
    (n-2)ab+\binom{n-2}{2}(a+b)+\binom{n-2}{3}
\right)
\cdot
T_{n-2}(a,b)
$$
\end{conj}
In order to obtain the conjecture about corners,
we just have to add $T_n(a,b)$ to $noc_n(a,b)$. So, $c_n(a,b)$ can be rewritten as follows:
$$
c_n(a,b)
=
\left(
    a^2 + b^2 + nab + \frac{(n^2-n-4)(a+b)}{2} + \frac{(n+2)(n-2)(n-3)}{6}
\right)
\cdot
T_{n-2}(a,b).
$$
Let $X(s)$ be the random variable counting the number of locations of a state $s$ of size $n$
of the PASEP, where a particle may jump to the right or to the left.
We can compute the conjectural expected value of $X$ by using the formula of Section~4 in~\cite{MR3459908},
\begin{align*}
\mathbb{E}(X)   &= \frac{1}{T_{n+1}(a,b)}\sum_{T\in\T_{n+1}}w(T)(2c(T)-1). \\
                &= \frac{
                    2\cdot(
                        a^2 + b^2 + (n+1)ab +
                        \frac{(n^2+n-4)(a+b)}{2} +
                        \frac{(n+3)(n-1)(n-2)}{6}
                    )
                }{
                    (a+b+n-1)(a+b+n-2)
                } -1\\
                &=\frac{
                    6[a^2 + b^2 + (n+1)ab] + 3(n^2+n-4)(a+b) + (n+3)(n-1)(n-2)              }{                   3(a+b+n-1)(a+b+n-2)
                } -1\\
                &= \frac{
                    3(a^2 + b^2) + 6nab + 3(n^2 - n - 1)(a + b) + n(n-1)(n-2)
                }{
                    (a+b+n-1)(a+b+n-2)
                }.
\end{align*}

Instead of studying $noc_n(a,b)$ as a sum over tree-like tableaux, we will study it
as a sum over non-occupied corners in $\T_n$. Let $noc$ be a non-occupied corner
of a tree-like tableau $T$, we define the weight of $noc$ as
$$w(noc):=w(T).$$
Let $NOC(\T_n)$ be the set of non-occupied corners in $\T_n$, we can rewrite
$noc_n(a,b)$ as
$$noc_n(a,b)=\sum_{noc\in NOC(\T_n)}w(noc).$$
The study of this conjecture brings a partitioning of non-occupied corners.
We denote by $NOC_{a,b}(\T_n)$ the set of non-occupied corners
with no point above them, in the same column, except in the first row
and no point at their left, in the same row, except in the first column.
The set of non-occupied corners with no point above them, except in the first row,
or no point at their left, except in the first column, but not both in the same time,
are denoted by $NOC_{a,1}(\T_n)$ and $NOC_{1,b}(\T_n)$ respectively.
The remaining corners are regrouped in $NOC_{1,1}(\T_n)$.
The different types of non-occupied corners are illustrated in Figure~\ref{fig:noc_partitioning}.
\begin{figure}[htbp]
    \begin{center}
        \captionsetup[subfigure]{width=100pt}
        \subfloat[$NOC_{a,b}(\T_n)$]{
            \label{fig:noc_ab}
            \includegraphics[scale=.35]{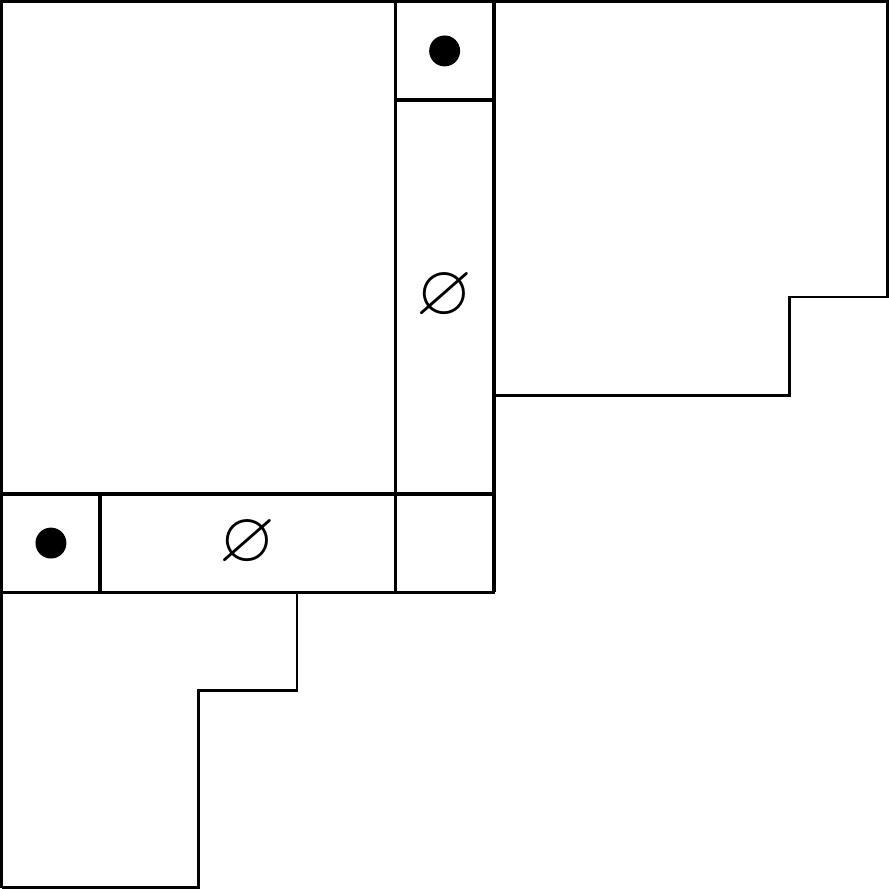}
        }
        \hfil\hfil
        \subfloat[$NOC_{a,1}(\T_n)$]{
            \label{fig:noc_a1}
            \includegraphics[scale=.35]{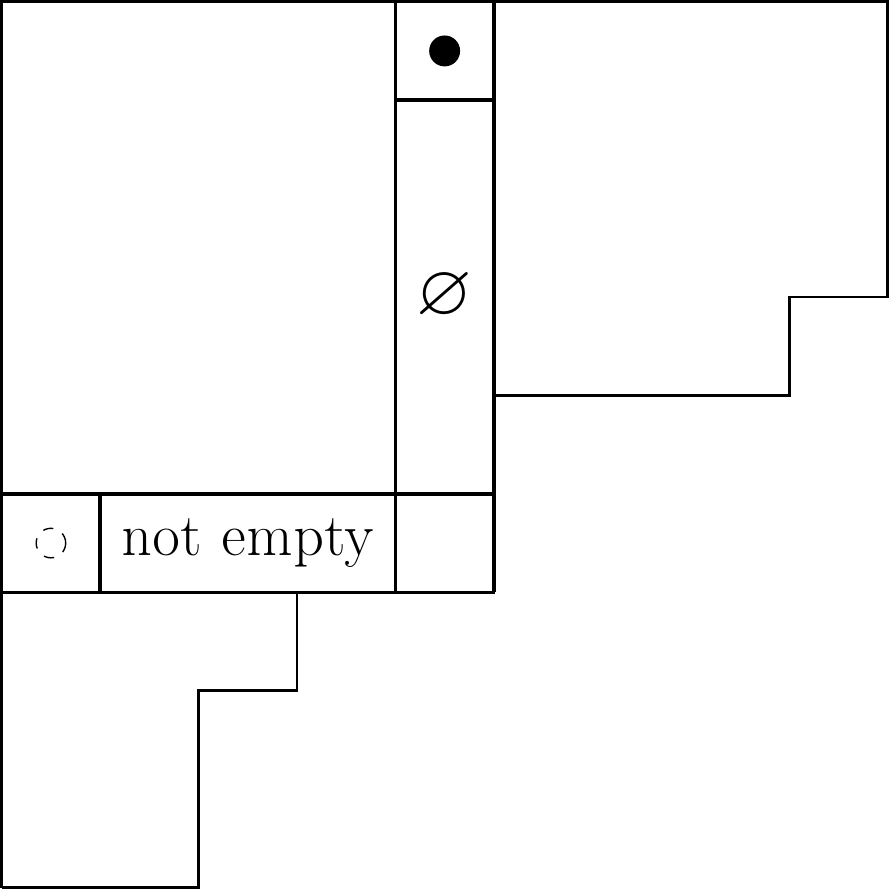}
        }
        \hfil\hfil
        \subfloat[$NOC_{1,b}(\T_n)$]{
            \label{fig:noc_1b}
            \includegraphics[scale=.35]{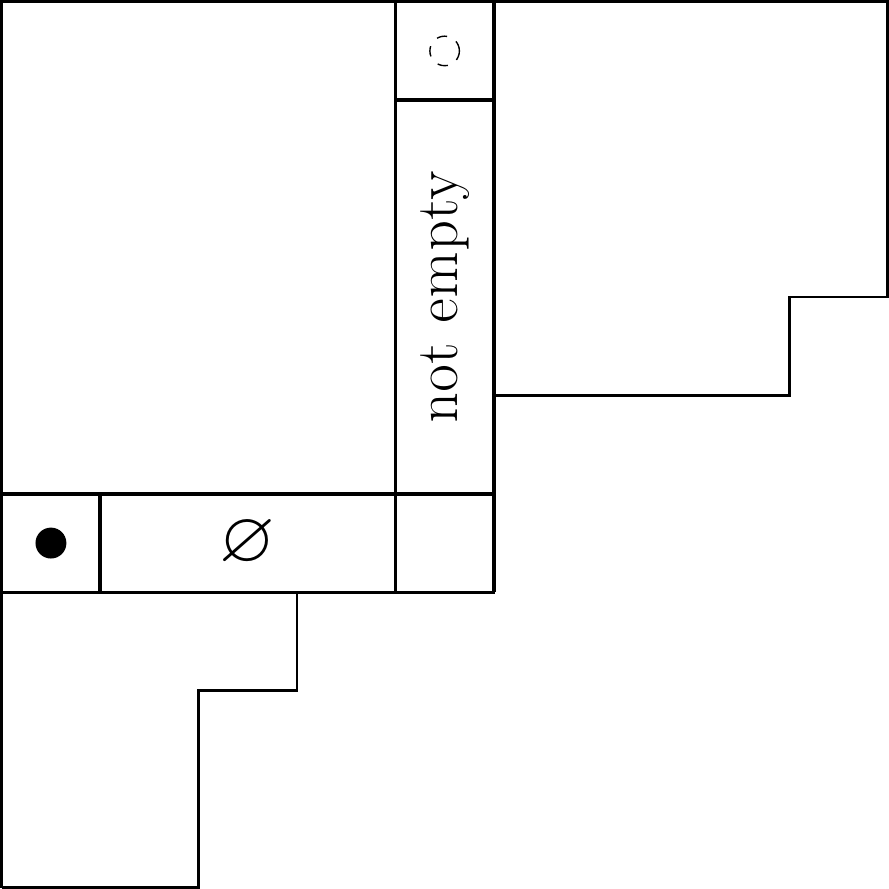}
        }
        \hfil\hfil
        \subfloat[$NOC_{1,1}(\T_n)$]{
            \label{fig:noc_11}
            \includegraphics[scale=.35]{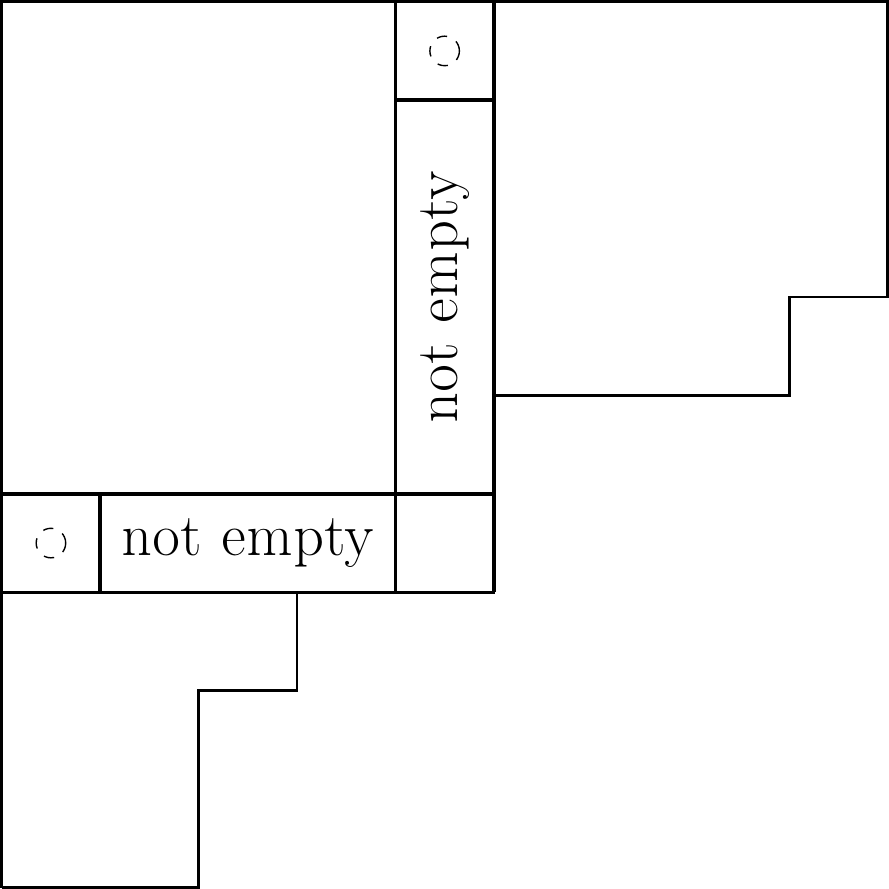}
        }
    \end{center}
    \caption{
        Partitioning of non-occupied corners in $\T_n$.
        ({\protect\tikz \protect\draw[dashed,color=black] (0,0) circle (.1);}
        means that the cell can be either empty or occupied.)
    }
    \label{fig:noc_partitioning}
\end{figure}
\begin{prop}
For $n\geq 3$, the ($a,b$)-analogue of the enumeration of $NOC_{a,b}(\T_n)$ is
$$
\sum_{noc\in NOC_{a,b}(\T_n)} w(noc)
=
(n-2) \cdot ab \cdot T_{n-2}(a,b)
.$$
\end{prop}
\begin{proof}
In order to show that, we put in bijection non-occupied corners of $\T_n$ of
this shape and the set of pairs $(T,i)$ where $T$ is a tree-like tableau of
size $n-2$ and $i$ is an interstice between two consecutive border edges of
$T$. Let $noc\in NOC_{a,b}(\T_n)$ and $T'$ be its tree-like tableau of size
$n$. Let $j$ be the integer such that $noc$ is the cell at the intersection of
row $j$ and column $j+1$. We obtain a tree-like tableau $T$ of size $n-2$ by
removing the row $j$ and the column $j+1$. The North-West corner $i$ of $noc$
corresponds to an interstice between two consecutive border edges of $T$, we
associate to $noc$ the pair $(T,i)$. Conversely, let us consider a pair
$(T,i)$. We construct a tree-like tableau $T'$ as follows, we add to $T$ a row
and a column ending a common cell $c$ with $i$ as its North-West corner, and we
had a point to the left-most (resp. highest) cell of the new row (resp.
column). In particular, $c$ is in $NOC_{a,b}(\T_n)$. The bijection is
illustrated in Figure \ref{fig:noc_ab_bij}.
\begin{figure}[htbp]
    \begin{center}
        \captionsetup[subfigure]{width=150pt}
        \subfloat[Bijection between $NOC_{a,b}(\T_n)$ and pairs $(T,i)$.]
            {
                \label{fig:noc_ab_bij}
                \includegraphics[scale=.4]{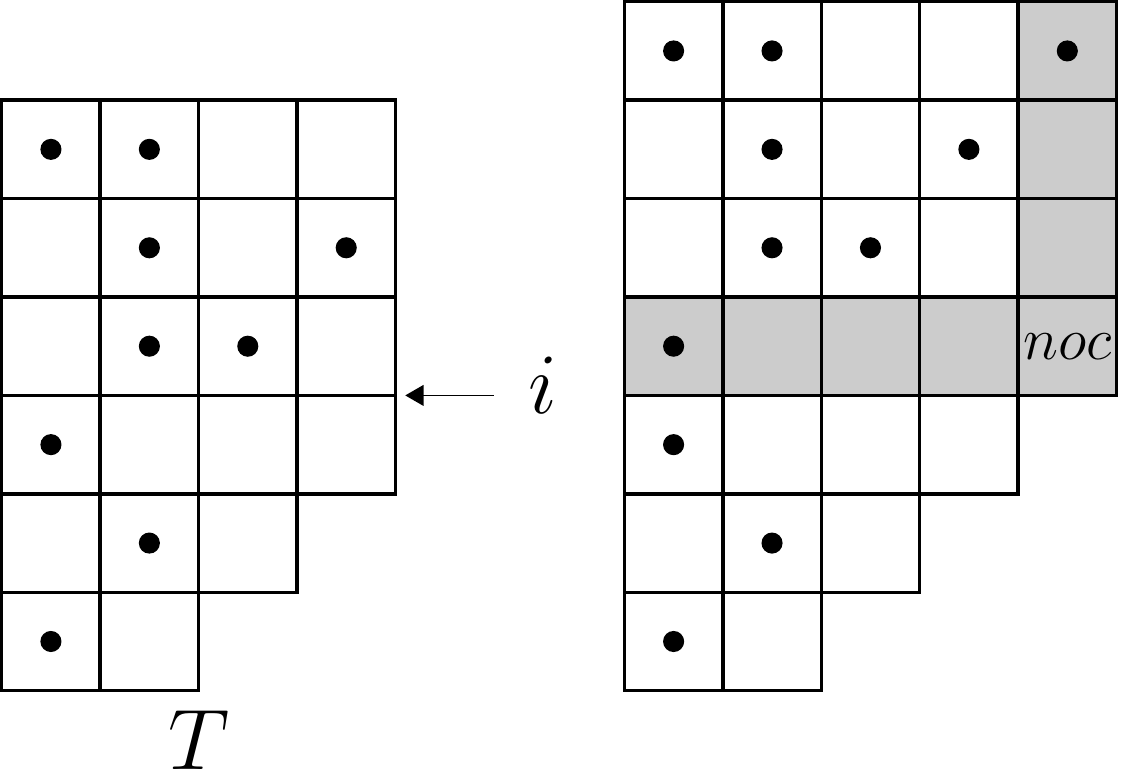}
            }
        \hfil\hfil
        \subfloat[
            Bijection between pairs $(T,r)$
            and $NOC_{a,b}(\T_n)\bigcup NOC_{a,1}(\T_n)$ .
        ]
            {
                \label{fig:noc_a1_bij}
                \includegraphics[scale=.4]{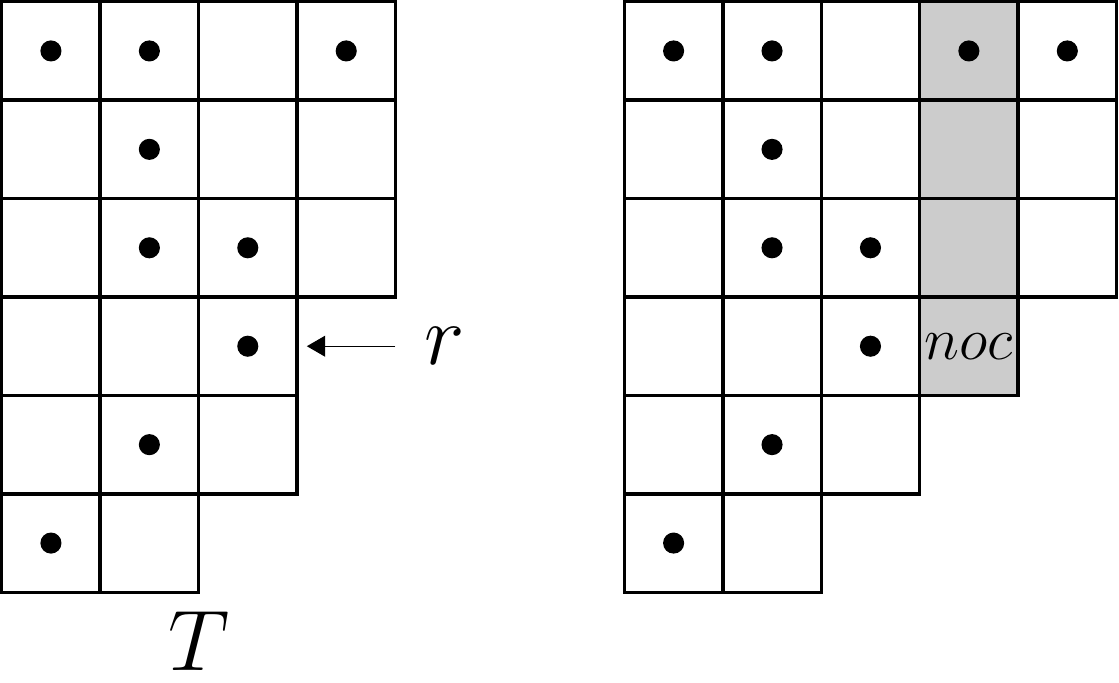}
            }
        \caption{}
        \label{fig:noc_a1_bijection}
    \end{center}
\end{figure}

For each tree-like tableau $T$ of size $n-2$, there are $n-2$ choices
of interstice $i$, in addition, the weight of $noc$ is equal to $ab\cdot w(T)$.
As a result,
$$
\sum_{
    noc\in NOC_{a,b}(\T_n)
} w(noc)
=
\sum_{
    \substack{
        T\in\T_{n-2} \\
        i \text{ interstice of } T
    }
} ab \cdot w(T)
=
(n-2) \cdot ab \sum_{T\in\T_{n-2}} w(T)
=
(n-2) \cdot ab \cdot T_{n-2}(a,b)
.$$
\end{proof}

We are also able to give an ($a,b$)-analogue of the enumeration of $NOC_{a,1}(\T_n)$
and $NOC_{1,b}(\T_n)$. In order to do that, we need an ($a,b$)-analogue
of the enumeration of tree-like tableaux of size $n$
with a fixed number of rows $k$,
$$
\A{n,k}
=
\sum_{
    \substack{T\in\T_n \\ T\text{ has }k\text{ rows }}
}
a^{top(T)}b^{left(T)}.
$$

From Proposition~3.4 of~\cite{MR3158273}, we already know
that $A_{1,1}(n,k)$ is the Eulerian number and satisfies
$$A_{1,1}(n+1,k) = kA_{1,1}(n,k) + (n+2-k)A_{1,1}(n,k-1).$$
In the general case, the linear recurrence satisfied by $\A{n,k}$ is
\begin{equation}\label{Euler_a_b}
\A{n+1,k}=(a-1+k)\A{n,k}+(b+n+1-k)\A{n,k-1}.
\end{equation}
As in \cite{MR3158273}, we consider the Eulerian polynomial:
$$A_n(t):=\sum_{k=1}^{n}\A{n,k}t^k.$$
\begin{lem}
For $n\geqslant 2$,
$$A_n(1) = T_n(a,b)\text{ and } A_n'(1) =(a+bn+\binom{n}{2}-1)T_{n-1}(a,b).$$
\end{lem}
\begin{proof}
The first identity is a consequence of the definition of $A_n(t)$.
For the second one,
we deduce from \eqref{Euler_a_b} that $A_n(t)$ satisfies the recurrence relation
\begin{equation*}
\begin{split}
A_n(t)&=(a-1)A_{n-1}(t)+tA'_{n-1}(t)+(b+n-1)tA_{n-1}(t)-t^2A'_{n-1}(t)\\
      &=(a-1)A_{n-1}(t)+(b+n-1)tA_{n-1}(t)+t(1-t)A'_{n-1}(t),
\end{split}
\end{equation*}
with initial condition $A_1(t)=t$.
Hence, by differentiating and evaluating at $t=1$,
we get the following recurrence relation for $A'_n(1)$
\begin{align*}
A_n'(1) &= (a-1)A'_{n-1}(1) + (b+n-1)(A'_{n-1}(1)+A_{n-1}(1)) - A'_{n-1}(1)\\
        &= (a+b+n-3)A'_{n-1}(1)+(b+n-1)A_{n-1}(1).
\end{align*}
For $n\geqslant 3$, dividing by $T_{n-1}(a,b)$, we get
$$\frac{A_n'(1)}{T_{n-1}(a,b)}=\frac{A'_{n-1}(1)}{T_{n-2}(a,b)}+(b+n-1).$$
Since $A'_2(1)=a+2b$, for $n\geqslant 2$,
$$A'_n(1)=(a+bn+\binom{n}{2}-1)(a+b+n-3)\cdots(a+b).$$
\end{proof}

We can now prove the following result,
\begin{prop}
For $n\geq 3$, the ($a,b$)-analogue of the enumeration of $NOC_{a,1}(\T_n)$ is
$$
f_n(a,b)
:=
\sum_{noc\in NOC_{a,1}(\T_n)} w(noc)
=
\binom{n-2}{2}\cdot a \cdot T_{n-2}(a,b)
,$$
and the ($a,b$)-analogue of the enumeration of $NOC_{1,b}(\T_n)$ is
$$
g_n(a,b)
:=
\sum_{noc\in NOC_{1,b}(\T_n)} w(noc)
=
\binom{n-2}{2}\cdot b \cdot T_{n-2}(a,b).$$
\end{prop}
\begin{proof}
In order to compute $f_n(a,b)$ we put in bijection elements of
$NOC_{a,1}(\T_n)\bigcup NOC_{a,b}(\T_n)$ with pairs $(T,r)$ where $T$ is a
tree-like tableaux of size $n-1$ and $r$ is a row of $T$, different from the
first one. Let $noc\in NOC_{a,1}(\T_n)$ and $T'$ its tree-like tableau of size
$n$. Let $j$ be the integer such that $noc$ is the cell at the intersection of
row $j$ and column $j+1$. We obtain $T$ by removing the column $j+1$, $r$
corresponds to the row $j$. It should be clear that this operation is
revertible. The bijection is illustrated in Figure~\ref{fig:noc_a1_bij}. Using
the previous notations, $w(noc)=a\cdot w(T)$. As a result,
\begin{align*}
f_n(a,b)    &= \sum_{
                  \substack{
                      T\in\T_{n-1}\\
                      r\text{ a row of $T$}
                  }
               } a\cdot w(T)
               -
               \sum_{noc\in NOC_{a,b}(\T_n)} w(noc) \\
            &= a \cdot \sum_{k=1}^{n-1} (k-1) \cdot \A{n-1,k}
               -
               (n-2) \cdot ab \cdot T_{n-2}(a,b) \\
            &= a \cdot (A'_{n-1}(1) - A_{n-1}(1)) - (n-2) \cdot ab \cdot T_{n-2}(a,b) \\
            &= a\left[(a+(n-1)b+\binom{n-1}{2}-1)-(a+b+n-3)-(n-2)b\right]\cdot T_{n-2}(a,b) \\
            &=a\binom{n-2}{2} \cdot T_{n-2}(a,b)
\\
\end{align*}

The axial symmetry with respect to the main diagonal
of the underlying diagram gives a bijection between
$NOC_{a,1}(\T_n)$ and $NOC_{1,b}(\T_n)$,
a top point becomes a left point and conversely.
Hence, $g_n(a,b)$ can be deduced from $f_n(a,b)$
by the identity $g_n(a,b)=f_n(b,a)$. Since $T_{n-2}(a,b)$
is a symmetric polynomial,
$g_n(a,b)=b\binom{n-2}{2} \cdot T_{n-2}(a,b)$.
\end{proof}

To prove the conjecture, we miss the ($a,b$)-analogue
of the enumeration of $NOC_{1,1}(\T_n)$.
The main issue is how to link these non-occupied corners with tree-like
tableaux of smaller size.
\subsubsection{A conjectural \texorpdfstring{$x$}{Lg}-analogue for symmetric tree-like tableaux}\label{sec:conjecture}

In the case of symmetric tree-like tableaux, $top$ and $left$ are always equal,
moreover, there is always a non-root point in the first row
and in the first column, therefore we will consider
$left^*(T)=left(T)-1.$
It gives a nice $x$-analogue of the enumeration of symmetric tree-like tableaux.
Indeed, Section~2.4 in~\cite{MR3158273} tells us that
$$
T_{2n+1}^{sym}(x)
:=
\sum_{T\in\mathcal{T}^{sym}_{2n+1}}x^{left^*(T)}
=
2^n\cdot(x+1)\cdots(x+n-1)
$$
(We believe that there is a mistake in Section~2.4 of~\cite{MR3158273},
we should take the definition
\[
 T^{sym}_{2n+1}(x,y,z):=\sum_{T\in \mathcal{T}^{sym}_{2n+1}} x^{left^ *(T)}y^{top^*(T)}z^{diag(T)},
\]
in order to get,
$$T^{sym}_{2n+1}(x,y,z)=(1+z)^n(x+y)(x+y+1)\cdots (x+y+n-2).)$$

It turns out that in the case of the enumeration of corners in symmetric
tree-like tableaux, the $x$-analogue might be nice as well. As in the
non-symmetric case, the $x$-analogue of the enumeration of occupied corners is
equal to $T_{2n+1}^{Sym}(x)$. Thus conjecturing an $x$-analogue for
non-occupied corners and unrestricted corner is equivalent. A computer
exploration using~\cite{stein2015sagemath1}, gives us the following expression:
\begin{conj}
The $x$-analogue of the enumeration of non-occupied corners in symmetric
tree-like tableaux is
\begin{multline*}
\sum_{T\in\mathcal{T}^{sym}_{2n+1}}noc(T)x^{left^*(T)}
=
\left[
    2nx^2 + 2(2n^2-4n+1)x + \frac{(n-2)(n-1)(4n-3)}{3}
\right] \cdot T_{2n-3}^{sym}(x)
.
\end{multline*}
\end{conj}
Using~\cite{stein2015sagemath1}, we can confirm this $x$-analogue until $n=7$:
$$
\begin{array}{cl}
    n=2,   &   4x+2\\
    n=3,   &   (6x^2+14x+6)*2\\
    n=4,   &   (8x^2+34x+26)*4(x+1)\\
    n=5,   &   (10x^2+62x+68)*8(x+2)(x+1)\\
    n=6,   &   (12x^2+98x+140)*16(x+3)(x+2)(x+1)\\
    n=7,   &   (14x^2+142x+250)*32(x+4)(x+3)(x+2)(x+1)\\
\end{array}
$$

In the non-symmetric case, we were only able to prove the coefficients of
$x^2$. It also corresponds to the empty corners such that the only point above
them is in the first row and only point to their left is in the first column.

\section{Conclusion and Remarks}
\label{sec:conclusion-remarks}

We computed the number of corners in (type $B$) permutation tableaux,
(symmetric) alternative tableaux and (symmetric) tree-like tableaux, by
interpreting the number of corners as a statistic on (signed) permutations.
Moreover, we gave a bijection between corners in tree-like tableaux and ascending
runs of size one in permutations. Finally, we partially proved a conjectural
($a$,$b$)-analogue and $x$-analogue of the enumeration of corners, in tree-like
tableaux and symmetric tree-like tableaux respectively.

It is worthy noting that the number of non-occupied corners in tree-like tableaux of size $n+1$ occurs in \cite[A005990]{sloane2011line}, which enumerates the total positive displacement of all letters in all permutations on $[n]$, i.e,
\begin{equation*}
\sum_{\pi\in \mathfrak{S}_n} \sum_{i=1}^n \max\{\pi_i-i,0\},
\end{equation*}
the number of double descents in all permutations of $[n-1]$ and also the sum of the excedances of all permutations of $[n]$. We say that $i$ is a double descent of a permutation $\pi=\pi_1\pi_2\cdots\pi_n$ if $\pi_i>\pi_{i+1}>\pi_{i+2}$, with {$1\leqslant i \leqslant n-2$} and an excedance if $\pi_i>i$, {with $1\leq i\leq n-1$}.
Besides, they are also related to coefficients of Gandhi polynomials, see~\cite{MR1902756}.
To find the relationship between both of them is also an interesting problem.

\acknowledgements
\label{sec:ack}
The authors would like to thank the anonymous referees for their valuable suggestions and comments that
have helped a lot to improve the quality and presentation of the paper.

A part of this research was driven by
computer exploration using the open-source software \texttt{Sage}~\cite{stein2015sagemath1}
and its algebraic combinatorics features developed by the
\texttt{Sage-Combinat} due to~\cite{stein2015sagemath2}.

\nocite{*}
\bibliographystyle{abbrvnat}

\label{sec:biblio}

\end{document}